\documentclass[11pt,a4paper]{article}

\usepackage{graphicx} 
\usepackage{a4wide} 
\usepackage[utf8]{inputenc}
\usepackage{amsfonts}
\usepackage{subcaption}
\usepackage{enumitem}
\usepackage{mathtools}
\usepackage[nottoc]{tocbibind}
\usepackage{bbm}
\usepackage{amsthm}
\usepackage{xfrac}
\usepackage{aligned-overset}
\usepackage[title]{appendix}
\usepackage{xcolor}
\usepackage{textcomp}
\usepackage{manyfoot}
\usepackage{booktabs}
\usepackage{listings}

\title{Stochastic nonlocal traffic flow models with Markovian noise}

\author{Timo Böhme\footnotemark[1],
\; Simone Göttlich\footnotemark[1],
\; Andreas Neuenkirch\footnotemark[1]}

\date{\today}

\begin{document}
\theoremstyle{plain}
\newtheorem{theorem}{Theorem}[section] 
\newtheorem{lemma}[theorem]{Lemma}
\newtheorem{proposition}[theorem]{Proposition}
\newtheorem{corollary}[theorem]{Corollary}

\theoremstyle{definition}
\newtheorem{definition}[theorem]{Definition}
\newtheorem{example}[theorem]{Example}

\theoremstyle{remark}
\newtheorem{remark}[theorem]{Remark}
\newtheorem{bemerkung}[theorem]{Bemerkung}
\newtheorem{assumption}[theorem]{Assumption}

\renewcommand{\thefootnote}{\fnsymbol{footnote}} 

\maketitle

\footnotetext[1]{University of Mannheim, Department of Mathematics, B6, 68159 Mannheim, Germany (timo.boehme@students.uni-mannheim.de, \{goettlich, neuenkirch\}@uni-mannheim.de).}

\begin{abstract}
\noindent
We extend the stochastic nonlocal traffic flow model from \cite{Boehme_2025} to more general random perturbations, including Markovian noise derived from a discretized Jacobi-type stochastic differential equation (SDE). 
Invoking a  deterministic stability estimate, we show that the arising random weak entropy solutions are measurable, ensuring that quantities such as the expectation are well-defined. 
We show that the proposed Jacobi-type noise is of particular interest as it ensures interpretability, preserves boundedness, and significantly alters the stochastic realizations compared to the previous white noise approach. 
Moreover, we introduce a local solution operator which provides information on the local effect of the noise and utilize it to derive a mean-value hyperbolic nonlocal PDE, which serves as a proxy for the mean value of the exact solution. 
The quality of this proxy and the impact of the noise process are analyzed in several simulation studies.
\end{abstract}



\section*{Introduction} 
In the search for mathematical models describing traffic flow, hyperbolic conservation laws have proven particularly effective over the past decades. Describing traffic as a macroscopic quantity, they are computationally efficient and allow for the simulation of large-scale networks. For a comprehensive overview, we refer to~\cite{Tosin2021_conference, Rosini2013} and to~\cite{garavellohanpiccoli2016book, GaravelloPiccoliBook} for the mentioned network extensions.
Although the aforementioned classical models assume that driving behavior depends on local quantities around each driver, nonlocal traffic flow models have become an active area of research.
As exemplarily (further) developed in \cite{BlandinGoatin2016, Chiarello2018, GoatinPuppo2024, ColomboGaravelloMercier2012, colombo2020local, Friedrich2018, Huang2022, keimer2018nonlocal}, 
these models extend local approaches by integrating traffic conditions (far-) ahead into the behavior.
In particular in the context of connected autonomous vehicles, this allows for the inclusion of not only local data but also remote information covering the entire road downstream.
From a modeling perspective, the use of such data enables the description of anticipatory drivers and smart vehicles that adjust their speed early in response to distant congestion, thus saving time and resources while stabilizing the traffic.\\
It is well known that driving behavior and traffic dynamics are subject to stochastic influences, which can be interpreted in two primary ways.
First, stochasticity may stem directly from agents, representing intrinsic randomness. For instance, human drivers do not have fixed reactions to specific traffic conditions, but rather follow a probabilistic response distribution within certain limits. Similarly, autonomous vehicles, which can react deterministically, are subject to measurement noise.
Second, from a modeling perspective, stochasticity can be understood as a placeholder for incomplete information regarding unknown variables.
In this context, even if the underlying traffic dynamics are deterministic, the modeling simplifications necessitate aggregating unobserved effects into a stochastic term.
Although some stochastic extensions exist for local macroscopic models, see e.g.,~\cite{Jabari2012,Li2012, Wen2025}, 
stochastic extensions to nonlocal models remain sparse.
To address this gap, in \cite{Boehme_2025}, we proposed a traffic flow model incorporating stochastic nonlocal velocities. The so-called sNV model belongs to the broader class of stochastic conservation laws with random fluxes, see e.g., \cite{Garnier2012,
Lions2013, 
Mishra2012,Risebro2015}.
Its base is a nonlocal model presented in~\cite{Friedrich2018} that uses downstream velocities but is evaluated using a noisy response function.
For this framework, in \cite{Boehme_2025}, we established theoretical results regarding the pathwise existence and uniqueness of weak entropy solutions and provided a time-dependent white-noise process satisfying the necessary regularity requirements.\\
The overall aim of this paper is to extend the initial framework, providing additional theoretical foundations and significantly expanding the numerical possibilities.
We proceed in three directions, which collectively yield a generalized and more capable sNV model.
First, we address open theoretical gaps from a stochastic perspective, including the stability of solutions and critically the existence of the expectation.
Second, we expand the model to incorporate autocorrelated Markovian noise, introducing transformed SDE-type noise of Jacobi nature.
This not only allows for a direct physical interpretation but leads to significantly stronger perturbations, as we demonstrate. 
Third, motivated by these strong perturbations, we derive a mean-value hyperbolic PDE based on a local solution operator to describe expected traffic conditions.
We support this with numerical Monte Carlo results that demonstrate convergence in the characteristic space. Thereby we simultaneously contribute to the fields of both stochastic conservation laws and nonlocal traffic flow modeling.\\
The structure is as follows: 
Section~\ref{sec:sNV_class} is dedicated to extending the theoretical foundations and capabilities of the sNV framework.
We begin by revisiting the deterministic nonlocal model from \cite{Friedrich2018}, alongside the stochastic extension and key results established in \cite{Boehme_2025}, in Sections~\ref{subsec:NV} through~\ref{subsec:ex_uniq}.
Building on this setup, we use a stability result regarding the noise-sensitivity of solutions, which enables us to present our first key contribution: the measurability of solutions, given in Section~\ref{subsec:well_posed_exp}.
We further develop the mathematical framework in Section~\ref{subsec:loc_sol_op} by deriving a local solution operator.
This operator is then utilized in Section~\ref{subsec:mean_velo} to establish our second core contribution: a mean-value PDE that captures the evolution of the expected value.
Next, in Section~\ref{subsec:markov_type_error}, we expand the initial white-noise approach to include autocorrelated noise and present a suitable SDE-type noise formulation along with the necessary transformations to ensure well-posedness and numerical implementability, constituting our third main contribution.
Section~\ref{sec:numerics} describes the numerical schemes used for the stochastic nonlocal models, along with the noise sampling established previously.
In addition, in Section~\ref{sec:num_results}, we employ these schemes to present numerical results for the SDE-type noise formulation and conduct Monte Carlo experiments related to the mean-value PDE.
While the strong perturbations introduced by the SDE-type noise motivate the need for a suitable proxy, the numerical analysis presented provides empirical evidence for the convergence of the mean-value approximation. 
The latter is further discussed in light of the beneficial influence of the nonlocal range, validating the proposed framework as a whole.
Finally, Chapter~\ref{sec:conc_outlook} summarizes our findings and offers a concluding discussion.

\section{The class of sNV models and their properties}\label{sec:sNV_class}

\subsection{A deterministic nonlocal velocity model (NV)}\label{subsec:NV}
The deterministic nonlocal velocity model, originally introduced in~\cite{Friedrich2018}, is given by the scalar conservation law
\begin{align}\label{eq:NV}\tag{NV}
	\partial_t\rho+\partial_x(\rho(W_\eta * v(\rho)))=0, \quad x \in \mathbb{R}, \; t>0,
\end{align}
where the convolution is 
\begin{align*}
	(W_\eta * v(\rho))(t,x):=\int_{x}^{x+\eta}  W_{\eta}(y-x) v(\rho(t,y)) \,dy, \quad \eta>0,
\end{align*}
and the Cauchy problem is equipped with initial conditions of the form
\begin{align}\label{eq:CP_BV}\tag{I}
	\rho(0,x)=\rho_0(x) \in (L^1 \cap \text{BV})(\mathbb{R};[0,\rho^{max}])
\end{align}
for $\rho^{max}>0$ given. The assumptions on the kernel function are as follows:
\begin{itemize}
\item[(A)] Given the look-ahead distance $\eta>0$, we assume that
$$ W_\eta \in C^1([0,\eta];\mathbb{R}^+) \;\;  \text{with} \;\;  W_\eta' \leq 0,\quad \int_{0}^{\eta} W_\eta(x) \,dx=W_0, \quad \lim_{\eta \to \infty} W_\eta(0)=0.$$
\end{itemize}
Regarding the velocity function, the following assumptions need to be made:
\begin{itemize}
\item[(B)]  Given $\rho^{max}>0$, we assume for the velocity $v(\rho)$ that
$$v \in C^2([0,\rho^{max}];\mathbb{R}^+) \;\; \text{with} \;\; v'\leq 0, \quad v(0)=v^{max}>0.$$
\end{itemize}
The solution concept for \eqref{eq:NV} are weak entropy solutions as of Kru\v{z}kov \cite{Kruzkov1970}. %
\begin{definition}[Nonlocal weak entropy solution]\label{def:nonlocal_weak_entropy_sol}\text{ }\\
	A function $\rho \in C([0,T];L^1(\mathbb{R}))$ with $\rho(t,\cdot) \in \text{BV}(\mathbb{R};\mathbb{R})$ is a weak entropy 
	solution to (\ref{eq:NV}) with (\ref{eq:CP_BV}), i.e., to the Cauchy problem, if
	\begin{align*}
		\int_{0}^{T} & \int_{-\infty}^{\infty} |\rho-c|  \partial_t \phi +\text{sign}(\rho-c)(f(t,x,\rho)-f(t,x,c)) \partial_x \phi -\text{sign}(\rho-c) \partial_x f(t,x,c) \phi \,dx \,dt \\
		+& \int_{-\infty}^{\infty} |\rho_0(x)-c|\phi(0,x) \,dx \geq 0
	\end{align*}
	holds for all non-negative test-functions $\phi \in C_0^1([0,T] \times \mathbb{R} ; \mathbb{R}^+)$ and any constant $c\in \mathbb{R}$.
	This reduces to \cite[Def. 1]{BlandinGoatin2016} for bounded $\rho$ and special choices of $\phi$ and $c$ \cite[2.12]{Friedrich2021_diss}.
\end{definition} %
\noindent
Under assumptions (A) and (B) it has been shown in \cite{Friedrich2021_diss} that equation \eqref{eq:NV} with initial condition \eqref{eq:CP_BV} admits a unique entropy solution. This result has been extended in \cite{Boehme_2025} to a randomly  perturbed velocity function.

\subsection{A stochastic nonlocal velocity model (sNV)}\label{subse:sNV}
As a first step towards combining a nonlocal velocity model with stochastic perturbations, we introduced a stochastic NV model in \cite{Boehme_2025}, given by the conservation law:
\begin{align}\label{eq:sNV}\tag{sNV}
    \partial_t\rho+\partial_x \Bigl(\rho\bigl(W_\eta*v_\epsilon(\rho,t)\bigr)\Bigr)=0.
\end{align}
Here, the convolution is  as above and stochasticity is introduced through a time-dependent velocity function
\begin{align*}
   v_\epsilon(\rho,t) = \max\{0, v(\rho) + \epsilon(t)\},
\end{align*}
where $\epsilon(t)$ denotes a time-dependent error term, which is given by
\begin{align*}
    \epsilon(t)(\omega)=\sum_{k=1}^{R_T} \epsilon^k(\omega) \chi_{[t^k,t^{k+1})}(t), \qquad t\in[0,T], \, \, \omega \in \Omega,
\end{align*}
where $t^{k}=k \delta_R$, $R_T=\lfloor T/\delta_R \rfloor$
and the random variables $\epsilon^k$, $k=1, \ldots, R_T$, are defined in a probability space $(\Omega, \mathcal{F}, \mathbb{P})$, and are independent as well as uniformly distributed on $[-\tau,\tau]$ for some $\tau \in (0, v^{max}]$. 
One of the main results of \cite{Boehme_2025} is that
for such a random perturbation the (sNV) model admits for every $\omega \in \Omega$ a unique entropy solution.
However, a revision of the proof shows that the independence assumption and the distribution assumption were  motivated by modeling reasons and were not required for this result. Thus, Theorem 5.8 of \cite{Boehme_2025} extends directly to random perturbations of the following type:
\begin{itemize}
    \item[(C)] Let $(\Omega, \mathcal{F}, \mathbb{P})$ be a complete probability space. The time-dependent error is given \begin{align}\label{eq:err_very_general}
    \epsilon(t)(\omega)=\sum_{k=1}^{R_T} \epsilon^k(\omega) \chi_{[t^k,t^{k+1})}(t), \qquad t\in[0,T], \, \, \omega \in \Omega, 
\end{align} where  $\epsilon^k$, $k=1, \ldots, R_T$,  are random variables on $(\Omega, \mathcal{F}, \mathbb{P})$, which are uniformly bounded by some $\tau \in (0,  v^{max}]$, that is
$$ \sup_{k=1, \ldots, R_T} |  \epsilon^k(\omega) | \leq \tau , \qquad \omega \in \Omega.$$
\end{itemize}
Under Assumption (C) we have $0 \leq v_\epsilon(\rho,t)  \leq 2  v^{max}$ and thus a deterministic maximal velocity. 
\begin{remark}
The assumption $\tau \leq v^{max}$ is a modeling assumption, which ensures that the noise does not overwhelm the system. Most of the  mathematical analysis of the SNV model given below is valid also under the global assumption $\tau <\infty$.
\end{remark}

\begin{remark}\label{rem:on_sND}
 While nonlocal models are known to provide more realistic driver behavior, the  velocity-based formulation positions the model within the framework of scalar conservation laws with stochastic fluxes.
 This enables the approximation scheme we develop in Section \ref{subsec:mean_velo} to leverage the existence of an expected flux function.
 Consider instead a stochastic nonlocal density model, 
 where we similarly perturb the nonlocal quantity, i.e., the density of the Blandin and Goatin model~\cite{BlandinGoatin2016}:
    \begin{align}\label{eq:sND} \tag{sND}
        \partial_t \rho + \partial_x\Bigl( \rho  v\bigl( W_\eta*
        \min \{ \max \{ \rho+\epsilon,0 \} ,\rho^{\max} \} \bigr)\Bigr)=0.
    \end{align}
Here, evaluating the expected flux depends entirely on the input to the deterministic flux function, which requires knowledge of $\mathbb{E}[\rho]$ throughout the nonlocal horizon.
As we describe in Section \ref{subsec:mean_velo}, this information is not directly available.
\end{remark}

\subsection{Existence, uniqueness and stability of the sNV model} \label{subsec:ex_uniq} 

\begin{theorem}[Existence, uniqueness and properties of (sNV)]\label{thm:ex_uniq_sNV_extended}\text{ }\\
   Let $\rho_0$ as in (\ref{eq:CP_BV}) and assume that assumptions (A), (B) and (C) hold.
    Then, for 
     any $T>0$ and any fixed $\omega \in \Omega$ a weak entropy solution $\rho(t,x)(\omega)$, in the sense of Def. \ref{def:nonlocal_weak_entropy_sol}, to 
     the Cauchy Problem of (\ref{eq:sNV}), i.e.,
     \begin{align*}
        \begin{cases}
			\partial_t \rho(t,x) + \partial_x f_{\epsilon}(t,x,\rho(t,x))=0, & (t,x) \in (0,T] \times \mathbb{R},  \\ 
			\rho(0,x) =\rho_{0}(x), & x \in \mathbb{R},
		\end{cases} 	
     \end{align*}
     with $f_{\epsilon}(t,x,\rho)=\rho \bigl(W_\eta * v_\epsilon(\rho,\cdot)\bigr)(t,x)$, exists and is unique.
    Further, it holds for all $\omega \in \Omega$:
    \begin{enumerate}
        \item[(1)] Maximum principle: for any $t \in [0,T]$ we have
        \begin{align*} 0 \leq \inf_{x \in \mathbb{R}} \rho_0(x) \leq \rho(t,x)(\omega) \leq \sup_{x \in \mathbb{R}} \rho_0(x) \leq \rho^{max}. \end{align*}
        \item[(2)] $L^1$-conservation: for any $t \in [0,T]$ we have
        \begin{align*} {\lVert (\rho(t,\cdot)) (\omega)\rVert}_{L^1(\mathbb{R})}={\lVert \rho_0 \rVert}_{L^1(\mathbb{R})}\quad \forall\, t \in [0,T]. \end{align*}
        \item[(3)] TV bounds: for any $T>0$ we have
        \begin{align*}   
            \text{TV}(\rho(T,\cdot)(\omega);\mathbb{R}) &\leq \exp \Bigl(T{C}_1 \Bigr) \text{TV}(\rho_0;\mathbb{R}),\\
             \text{TV}(\rho(\cdot,\cdot)(\omega);\mathbb{R} \times [0,T]) &\leq T \exp \bigl(T C_1\bigr)  C_2
                \text{TV}(\rho_0;\mathbb{R}),
        \end{align*}
        where $C_1=C_1(W_\eta,v,\rho^{max},\tau)>0$ and  $C_2=C_2(W_\eta,v,\rho^{max},\tau)>0$ are constants, which only depend on $W_{\eta}$, $v$, $\rho^{max}$ and $\tau$.
    \end{enumerate}
\end{theorem}
\noindent
As noted above, this result can be derived along the same lines as Theorem 5.8 in \cite{Boehme_2025}. This proof is carried out using a Godunov-type approximation $(\rho_j^n)_{n=0, \ldots, N_T, j \in \mathbb{Z}}$. While the 
quantities $\rho_j^n$ are $(\Omega, \mathcal{F})$-$(\mathbb{R},\mathcal{B}(\mathbb{R}))$ measurable, i.e., they are well-defined random variables, the weak entropy solution $\rho$ is obtained using a sub-sequence argument and Helly's theorem for fixed $\omega \in \Omega$. Thus, from this approach, it remains unclear whether quantities as
$$ \int_{\mathbb{R}} \mathbb{E}[ \rho(t,x)^p] \phi(x) dx, \qquad t \in [0,T],$$
with $p \geq 0$ and $\phi \in L^{\infty}(\mathbb{R};\mathbb{R}^+)$
or $  \mathbb{E}[ \rho(t,x)^p], \; t \in [0,T], \, x \in \mathbb{R}$, 
with $p>0$ are well defined. 
The key to solving these technical issues is the following deterministic stability result assuming given noise realizations.

\begin{lemma}[Stability estimate]\label{lem:stab_result}\text{ }\\
Let $\gamma_1,\gamma_2:[0,T] \rightarrow \infty$ be given by
 $$ \gamma_i(t)=\sum_{k=1}^{R_T} \gamma_i^k \chi_{[t^k,t^{k+1})}(t), \qquad t\in[0,T],$$
 where $\gamma_i^{k} \in [-\tau,\tau]$, $k=1, \ldots R_T$, $i=1,2$, with $\tau \in (0,v^{max}]$. Let $\rho_1^0,\rho_2^0$ be as in (\ref{eq:CP_BV}) and assume that assumptions (A) and (B) hold.
Finally, let $\rho_1$ and $\rho_2$ be
the respective weak solutions to \eqref{eq:sNV} in the sense of Def. \ref{def:nonlocal_weak_entropy_sol}, i.e., they are the solutions to 
\begin{align*}
        &\partial_t \rho_1(t,x) + \partial_x \bigl(\rho_1(t,x) V_{{\gamma_1}}(t,x)\bigr)=0,  && V_{{\gamma_1}}=W_\eta*v_{{\gamma_1}}(\rho_1,t),  &\rho_1(0,x)=\rho_1^0(x),  \\
        &\partial_t \rho_2(t,x) + \partial_x \bigl(\rho_2(t,x) V_{{\gamma_2}}(t,x)\bigr)=0,  && V_{{\gamma_2}}=W_\eta*v_{{\gamma_2}}(\rho_2,t), &\rho_2(0,x)=\rho_2^0(x).  
    \end{align*}
Then, we have
\begin{align*}
    {\left\lVert \rho_1(t,\cdot)-\rho_2(t,\cdot) \right\rVert}_{L^1(\mathbb{R})} \leq  \exp(KT \|v'\|_{\infty}) \left( {\left\lVert \rho_1^0-\rho_2^0 \right\rVert}_{L^1(\mathbb{R})} +K \int_0^T \left\lvert \gamma_1(t) - \gamma_2(t) \right\rvert \,dt \right)
\end{align*}
with 
$$ K=
W_{\eta}(0) \left( 2 \| \rho_1^0 \|_{L^1(\mathbb{R})} +  \exp(TC_1)TV(\rho_0;\mathbb{R}) \right)+  \|W'_{\eta} \|_{\infty} \| \rho_1^0 \|_{L^1(\mathbb{R})} . $$
\end{lemma}
\noindent
Since $\gamma_1$ and $\gamma_2$ are arbitrary but deterministic, to establish the above result, we can proceed similarly to classical uniqueness proofs with the additional aspect of differing velocity functions. This is captured by the Lipschitz estimate in both variables:
 \begin{align*}
    \left\lvert v_{\gamma_1}(\rho_1)-v_{\gamma_2}(\rho_2)\right\rvert \leq 
    \lVert v' \rVert_{\infty}  \left\lvert \rho_1 - \rho_2 \right\rvert 
    + \left\lvert \gamma_1 - \gamma_2 \right\rvert, 
\end{align*}
which carries through the subsequent estimates. As major parts of the proof follow classical techniques, see e.g., \cite{BlandinGoatin2016, Chiarello2018, Friedrich2018}, the full derivation is provided in Appendix \ref{subsec:proof_lemma_stab}.
  
\subsection{Measurability of solutions}\label{subsec:well_posed_exp}
We now transition from deterministic perturbations to random noise with corresponding solutions, presenting the first of our three key contributions.
Similarly, as described in \cite[2.2]{Barth_2016} and \cite[3.11]{Mishra2012}, we establish the measurability of solutions to \eqref{eq:sNV} using the stability result above.

\begin{lemma}[Measurability of solutions]\label{lem:mb_sols}\text{ }\\
   Let assumptions (A), (B) and (C) hold, and let $\rho$ be the weak entropy solution from Theorem \ref{thm:ex_uniq_sNV_extended}. Moreover, let $t \in [0,T]$. 
Then, for all  $x \in \mathbb{R}$ there exist random variables
   $$ \rho^*(t,x):(\Omega, \mathcal{F}) \rightarrow (\mathbb{R}, \mathcal{B}(\mathbb{R}))$$
   such that
   $$  \mathbb\rho^*(t,x)(\omega)=\rho(t,x)(\omega) $$
for almost all $x \in \mathbb{R}$ and almost all $\omega \in \Omega$. In other words, $\rho^*(t,\cdot)(\cdot)$ is a measurable modification of $\rho(t,\cdot)(\cdot)$.
\end{lemma}
\noindent
Since the underlying probability space is complete, i.e., the sigma-algebra $\mathcal{F}$ contains all $\mathbb{P}$-zero sets, the same holds for the product space $$( \mathbb{R} \times \Omega ,  \mathcal{B}(\mathbb{R})\otimes\mathcal{F} , \lambda \otimes \mathbb{P}).$$ Thus, for the entropy solution for any fixed $t \in [0,T]$,
the mappings
$$ \rho(t,\cdot)(\cdot):  \mathbb{R} \times \Omega  \rightarrow \mathbb{R}$$ are   $\mathcal{B}(\mathbb{R})\otimes\mathcal{F}- \mathcal{B}(\mathbb{R}) $ measurable.
This yields now that quantities as 
$$ \int_{\mathbb{R}} \phi(x) \mathbb{E}[ \rho(t,x)^p] dx$$
with $p > 0$ and $\phi \in L^{\infty}(\mathbb{R};\mathbb{R}^+)$
or $$  \mathbb{E}[ \rho(t,x)^p],$$
with $p > 0$ are well-defined for all $t \in [0,T]$ and almost all $x \in \mathbb{R}$.

\begin{proof}
\textbf{(1)}
  Assume now that the process $\epsilon$  takes only finitely many values. Thus, the set of possible values of  $\epsilon$ can be defined as 
  \begin{align*}
      \{ \gamma_{1}, ..., \gamma_{L} \}, \quad L \in \mathbb{N},
  \end{align*}
  with $\gamma_i$ as in Lemma \ref{lem:stab_result}.
  This allows us to assign for $i \in I = \{1,...,L\}$ the pre-image subsets 
  \begin{align*}
    \Omega_i:= \{ \omega \in \Omega : \epsilon(\cdot)(\omega)=\gamma_i \} \in \mathcal{F}, \quad i \in I.
  \end{align*}
For each $\omega \in \Omega_i$ we obtain the same entropy solution
$ \rho(t,\cdot)(\omega)$ that we will denote by $\rho_{\gamma_i}(t,\cdot)$. Since
$\rho_{\gamma_i}(t,\cdot) \in L^1(\mathbb{R})$ we have that
$$A_i:=\{x \in \mathbb{R}: \rho_{\gamma_{i}}(t,x)\leq \alpha \} \in \mathcal{B}(\mathbb{R})$$
for every $\alpha \in \mathbb{R}$. We have 
  \begin{align*}
      \{ (x,\omega)\in  \mathbb{R} \times \Omega : \rho(t,x)(\omega) \leq \alpha \} 
      &=\bigcup_{i \in I} \{ (x,\omega)\in  \mathbb{R} \times \Omega_i : \rho(t,x)(\omega) \leq \alpha \}\\
      &=\bigcup_{i \in I} \{x \in \mathbb{R}: \rho_{\gamma_{i}}(t,x)\leq \alpha\} \times \Omega_i \\ & =\bigcup_{i \in I} A_i \times \Omega_i \in  \mathcal{B}(\mathbb{R}) \otimes \mathcal{F}
  \end{align*}
due to the finiteness of $I$.
Thus, for a finite-dimensional error process, the weak entropy solution
$$ \rho(t,\cdot)(\cdot):  \mathbb{R} \times \Omega  \rightarrow \mathbb{R}$$ is   $\mathcal{B}(\mathbb{R})\otimes\mathcal{F}- \mathcal{B}(\mathbb{R}) $ measurable for any $t \in [0,T]$.\\

\noindent \textbf{(2)} For the general case, we use an approximation argument:\\

   \noindent \textbf{(2a)} Since the random variables $\epsilon^k$ are uniformly bounded by $\tau$, using quantization, see, e.g., Theorem 5.2(b) in \cite{pages}, we can find discrete random variables: 
   \begin{align*}
    \epsilon_{(N)}^1, \, ... \, , \epsilon_{(N)}^{R_T}, 
  \end{align*} which take at most $N \in \mathbb{N}$ values and are uniformly bounded by $\tau$,
  such that
  \begin{align*}
      \mathbb{E} \left[ \max_{k=1,..,R_T} \lvert \epsilon_{(N)}^k-\epsilon^k \rvert^2 \right] \leq C(\tau,R_T)^2 N^{-2}
  \end{align*}
  for a constant $C(\tau,R_T)>0$.
  For the corresponding error process
  \begin{align*}
      \epsilon_{(N)}(t;\omega)=\sum_{k=1}^{R_T} e^k_{(N)}(\omega) \chi_{[t_k, t_{k+1})}(t)
  \end{align*}
  we have 
  \begin{align*}
      \mathbb{E} \int_0^T \lvert \epsilon_{(N)}(t)-\epsilon(t) \rvert \,dt \leq T \sqrt{C(\tau,R_T)} \cdot N^{-1}
  \end{align*} 
  using the Lyapunov inequality.
  In particular, we have 
  \begin{align}\label{eq:conv_gep_N}
      \mathbb{E} \int_0^T \lvert \epsilon_{(N+M)}(t)-\epsilon_{(N)}(t) \rvert \,dt \leq T \sqrt{C(\tau,R_T)} \cdot N^{-1}
  \end{align} 
  for all $N,M \in \mathbb{N}$.

\noindent \textbf{(2b)} Denote now by $\rho_{(N)}$ the unique weak entropy solution corresponding to $\epsilon_{(N)}$. Let $t \in [0,T]$. Since the constant $K$ in Lemma \ref{lem:stab_result} is deterministic,  equation \eqref{eq:conv_gep_N} gives
$$ \mathbb{E} \| \rho_{(N+M)}(t;\cdot) - \rho_{(N)}(t;\cdot) \|_{L^{1}(\mathbb{R})} \leq K \exp( KT \|v'\|_{\infty})T \sqrt{C(\tau,R_T)} \cdot N^{-1}
$$
and therefore $ \rho_{(N)}(t,\cdot)(\cdot)$ is a Cauchy-sequence in $L^{1}(\mathbb{R} \times \Omega)$.
By completeness there exists a measurable $\rho^*(t,\cdot)(\cdot) \in L^{1}(\mathbb{R} \times \Omega)$ such that
$$ \mathbb{E} \| \rho_{(N)}(t;\cdot) - \rho^*(t;\cdot) \|_{L^{1}(\mathbb{R})} \rightarrow 0, \qquad N \rightarrow \infty,
$$
as well
as \begin{align} \label{rho*-conv}
 \| \rho_{(N_{\ell})}(t,\cdot)(\omega) - \rho^*(t,\cdot)(\omega) \|_{L^1(\mathbb{R})} \rightarrow 0, \qquad \ell \rightarrow \infty, \end{align}
for almost all $\omega \in \Omega$ for a sub-sequence $(\rho_{(N_{\ell})}(t,\cdot)(\cdot))_{\ell  \in \mathbb{N}}$ of  $(\rho_{(N)}(t,\cdot)(\cdot))_{N \in \mathbb{N}}$.

\noindent \textbf{(2c)} Using the discrete structure of $\epsilon$ as well as Boole's and Markov's inequality, we have
  \begin{align*}
      \sum_{N=1}^\infty \mathbb{P}\Bigl( \max_{k=1,..,R_T} \lvert \epsilon_{(N)}^k-\epsilon^k\rvert \geq \psi\Bigr) &\leq
      \sum_{N=1}^\infty \sum_{k=1}^{R_T} \mathbb{P}\Bigl(\lvert \epsilon_{(N)}^k-\epsilon^k\rvert \geq \psi\Bigr)\\
      &\leq \sum_{N=1}^\infty \sum_{k=1}^{R_T} \dfrac{ \mathbb{E} \left[\lvert \epsilon_{(N)}^k-\epsilon^k \rvert^2 \right]}{\psi^2} \\ &\leq  \sum_{N=1}^\infty \dfrac{R_T C N^{-2}}{\psi^2} < \infty.
  \end{align*}
  Therefore, we can use the Borel-Cantelli lemma and conclude that
  \begin{align*}
      \max_{k=1,..,R_T} \lvert \epsilon^k_{(N)}(\omega)-\epsilon^k(\omega) \rvert \rightarrow 0,\; N\rightarrow\infty,
  \end{align*}
  for almost all $\omega \in \Omega$.
  By the boundedness of the noise terms and the dominated convergence theorem, it follows that
  \begin{align*}
      \int_0^T \lvert \epsilon_{(N)}(t)(\omega)-\epsilon(t)(\omega) \rvert \,dt \rightarrow 0,\qquad N\rightarrow\infty, 
  \end{align*}
   for almost all $\omega \in \Omega$.
Lemma \ref{lem:stab_result}  and the previous estimate now imply that
\begin{align}\label{rho-conv} \| \rho_{(N_{\ell})}(t;\cdot)(\omega) - \rho(t;\cdot)(\omega) \|_{L^{1}(\mathbb{R})} \rightarrow 0, \qquad \ell \rightarrow \infty, \end{align}
for almost all $\omega \in \Omega$. 
 Equations \eqref{rho*-conv} and \eqref{rho-conv}  now yield the following result
  \begin{align*}
&  \| \rho^*(t,\cdot)(\omega)-  \rho(t,\cdot)(\omega) \|_{L^1(\mathbb{R})} \\ & \quad \leq \lim_{\ell \rightarrow \infty} \| \rho^*(t,\cdot)(\omega)-  \rho_{(N_{\ell})}(t,\cdot)(\omega) \| _{L^1(\mathbb{R})}+ \lim_{\ell \rightarrow \infty} \| \rho(t,\cdot)(\omega)-  \rho_{(N_{\ell})}(t,\cdot)(\omega) \| _{L^1(\mathbb{R})} =0
  \end{align*} 
  for almost all $\omega \in \Omega$.
  Consequently, we must have that
$$ \rho^*(t,x)(\omega)=  \rho(t,x)(\omega) $$
for almost all $x \in \mathbb{R}$ and  almost all $\omega  \in \Omega$.
 \end{proof} 

\begin{remark}
The limit $\rho^*$ does not depend on the choice of the approximation sequence  $\epsilon_{(N)}$.
For  two different $L^1$-approximations $\epsilon_{(N),1}$, $\epsilon_{(N),2}$ of $\epsilon$, 
Lemma \ref{lem:stab_result} again implies that  
\begin{align*}
     &  \lVert \rho_{{(N),1}}(t,\cdot;\omega) -         \rho_{{(N),2}}(t,\cdot;\omega) \rVert_{L^1(\mathbb{R})} 
      \leq 
      K \exp( KT \|v'\|)
      \int_0^T   \lvert \epsilon_{(N),1}(t)-\epsilon_{(N),2}(t) \rvert \,dt \rightarrow 0
\end{align*}
for $N \rightarrow\infty$.
\end{remark}
\begin{remark}
    As already mentioned, we employed in \cite{Boehme_2025} a Godunov-type approximation $\rho^{\Delta x}(t,\cdot)$, which converges for fixed $t \in [0,T]$ under a CFL-condition along a suitable  sub-sequence in $L^1_{\textrm{loc}}(\mathbb{R})$ to $\rho(t,\cdot)$. Using dominated convergence, the same holds true now for the convergence of $ \mathbb{E} [\rho^{\Delta x}(t,\cdot)^p]$ to $ \mathbb{E}[\rho(t,\cdot)^p]$ in $L^1_{\textrm{loc}}(\mathbb{R})$ for any $p \in \mathbb{R}$. 
For more details on the Godunov scheme see Section \ref{sec:numerics}.
\end{remark}

\subsection{Local solution operator}\label{subsec:loc_sol_op}

Now let us consider the  parametrized deterministic NV model
\begin{align*}
    \partial_t \xi +\partial_x \Bigl(\xi \bigl(W_\eta*v_a(\xi,t)\bigr)\Bigr)=0
\end{align*}
with initial value and modified velocity function:
\begin{align*}
   \xi(0,x)=\xi_0(x),\qquad v_a(\xi,t) = \max\{0, v(\xi) + a\},
\end{align*}
where $|a| \leq v_{max} $. Under assumptions (A), (B) and \eqref{eq:CP_BV}, this equation has a unique weak entropy solution $\xi^{a}$ and we can define
the operators
\begin{align}\label{eq:sol_operator}
    \mathcal{S}_{t}^{a}\bigl[\xi_0\bigr](\cdot):=\xi^{a}(t,\cdot), \qquad t \in [0,T]. 
\end{align}
In our stochastic model we have piecewise constant noise in time, such that we can use these operators also to describe the evolution of our stochastic NV-model.
Using the time-grid from \eqref{eq:err_very_general} and setting $$\gamma_k=\gamma(k \delta_R), \quad k=0, \ldots, R_T,$$
we have
$$ \rho(t_k+\delta t,\cdot)=\mathcal{S}_{\delta t}^{\gamma_k}\bigl[\rho(t_k,\cdot)\bigr](\cdot), \quad \delta t \in [0, \delta_R], \,\, k=0,1, \ldots R_T-1.$$
Thus, locally the evaluation of $\rho$ can be described by the solution operators $\mathcal{S}_{\delta t}^a$ with $\delta t \leq \delta_R$, which we therefore call {\it local solution operators}.
Lemma \ref{lem:stab_result} shows that these operators are locally Lipschitz in $\xi$ and $a$. We have
\begin{align*} \| \mathcal{S}_{t}^{a_1}\bigl[\xi_1\bigr]- \mathcal{S}_{t}^{a_2}\bigl[\xi_2\bigr] \|_{L^{1}(\mathbb{R})}
\leq  \exp(Kt \|v'\|_{\infty}) \left( {\left\lVert \xi_1-\xi_2 \right\rVert}_{L^1(\mathbb{R})} +K t |a_1-a_2| \right),
\end{align*}
with 
$$ K=
 W_{\eta}(0) \left( 2 \| \xi_1 \|_{L^1(\mathbb{R})} +  \exp(tC_1)TV(\xi_1;\mathbb{R}) \right)+  \|W'_{\eta} \|_{\infty} \| \xi_1 \|_{L^1(\mathbb{R})} . $$

\subsection{Mean velocity function and mean value proxy}\label{subsec:mean_velo}

If we want to estimate quantities as the expected density
$ \mathbb{E}[\rho(t,x)]$,
this can be done using the Godunov-scheme and the standard Monte-Carlo approach, i.e., we sample $M$ i.i.d. copies of $\rho^{\Delta x}(t,x)$ and average these, i.e.,
\begin{align*}
    \mathbb{E}[\rho(t,x)] \approx \frac{1}{M} \sum_{k=1}^M \rho^{\Delta x, (k)}(t,x),
\end{align*}
where the superscript $(k)$ denotes the i.i.d. copies. 
Since this approach is computationally demanding, one might look for other approximations of the mean value.\\
In general, the nonlinear coupling between $\rho$ and $\epsilon$ prevents the direct reconstruction of the expected density $\mathbb{E}[\rho(t,x)]$ from the statistical properties of $\epsilon$ alone. 
In particular, the stochastic flux 
\begin{align*}
f_{\epsilon}(t,x,\rho)=\rho \bigl(W_\eta * v_\epsilon(\rho,\cdot)\bigr)(t,x)
\end{align*}
depends on $\rho$ in a nonlocal manner, which not only prevents us from obtaining a closed form equation for $\mathbb{E}[\rho(t,x)]$, but makes numerical evaluation using, e.g., collocation methods highly expensive. 
In terms of kinetic theory this issue is often referred to as the Closure-Problem. Whereas for example in \cite{Abdelmalik2016, Levermore_1996} minimization concepts are used to obtain closed equations in an approximate but analytical way,
\cite{cui2025} identifies expected drift-parameters by observing conditional trajectories on short time-intervals, leading to empirical results on the governing expectation law.
In our case, we leverage the fact that stochasticity in the sNV model is introduced solely through the flux, whereas the PDE mechanics remain deterministic.
In detail, for a given constant realization of the noise $a \in [-\tau,\tau]$ the density evolution is fully described by $\mathcal{S}_t^a$ as of \eqref{eq:sol_operator}. 
Hence, if the distribution $\nu_t$ of $\epsilon(t)$ is known, a natural candidate for the mean density is obtained by averaging these deterministic evolutions starting from the known initial value:
\begin{align*}
\tilde{m}(t,x) := \int_{-\tau}^{\tau} \mathcal{S}_t^{a}[\rho_0](x) \, \nu_t(da), 
\end{align*}
However, this is not directly computable as it requires knowledge of the solution operator for every $a$. Moreover, this equality can not be calculated for arbitrary $t \in (0,T]$. In general, we have
$$ \mathbb{E} \big [\rho(t_k+\delta t,\cdot) \ | \, \rho(t_k, \cdot) \big]=
\int_{-\tau}^{\tau} \mathcal{S}_t^{a}[\rho(t_k, \cdot)](x) \, \nu_t(da), \quad\delta t \in [0, \delta_R], $$
only if $\epsilon(t)$ is independent from $\rho(t_k,\cdot)$ for $t \in (t_k,t_{k}+\delta_R]$ and does not lead to an accessible expression for the mean density due to the non-linear nature of the local solution operator.\\

Instead, we analyze the deterministic proxy $\bar{m}(t,x)$ given by the NV model:
\begin{align}\label{eq:EsNV}\tag{EsNV}
     \partial_t \bar{m} +\partial_x\Bigl( \bar{m} \bigl( W_\eta *\bar{v}(\bar{m},t)\bigr) \Bigr)=0,
\end{align}
with initial condition $\bar{m}(0,x) := \rho_0(x)$, and the expected velocity $\bar{v}$ is defined by
\begin{align}\label{eq:eyp_density}
\bar{v}(\bar{m},t) := \int_{-\tau}^{\tau} v_a(\bar{m},t)  \nu_t(da),
\end{align}
where $\nu_t$ denotes the distribution of $\epsilon(t)$. For many error processes, $\bar{v}$ can be computed explicitly, see below.
Essentially, we are decoupling the noisy propagation from the perturbed densities and are only propagating \emph{one} expected density, using an averaged velocity function. This can also be understood as using the approximation:
\begin{align}\label{eq:exp_approx}
    \mathbb{E} \bigl[ \rho(t,x) {V}_{\epsilon}(t,x;\rho(t,\cdot)) \bigr] 
    \approx \mathbb{E} \left[ \rho(t,x)  \right]  \mathbb{E} \bigl[ {V}_{\epsilon}(t,x; \mathbb{E}  [\rho(t,\cdot)]) \bigr].
\end{align}

\begin{remark}
    As mentioned in Remark \ref{rem:on_sND}, we note that a similarly perturbed model like \eqref{eq:sND} does not exhibit a separable error contribution as captured in \eqref{eq:exp_approx}. This prevents the calculation and use of an expected velocity, which is an advantage of the  formulation \eqref{eq:sNV}.
\end{remark}
We will study the validity of the approximation aforementioned in our computations carried out for a high-noise-autocorrelated error process in Section \ref{sec:num_results}. 
Before introducing such a specific class of processes and related numerical schemes in Section \ref{subsec:markov_type_error}, we illustrate the approximation capabilities of \eqref{eq:EsNV} by comparing its characteristics with those of \eqref{eq:sNV}.

\begin{definition}[Characteristics of \eqref{eq:sNV} and \eqref{eq:EsNV}]\label{def:char}\text{}\\
Let 
    \begin{align*}
    V_\epsilon[\rho](t,x)&:=(W_\eta * v_\epsilon(\rho,\cdot))(t,x), \quad \text{and}\quad
    \bar{V}[\bar{m}](t,x):=(W_\eta * \bar{v}(\bar{m},t))(t,x).
    \end{align*}
For any fixed $\omega \in \Omega$, let $\rho$ be the unique weak entropy solution of (sNV) from Theorem \ref{thm:ex_uniq_sNV_extended}. The stochastic characteristics $X_{\rho,\epsilon}:[0,T] \times \mathbb{R} \times [0,T] \to \mathbb{R}$ are defined as the solutions of the integral equation
\begin{align}\label{eq:chars_integral_eq}
    X_{\rho,\epsilon}[t_0,x_0](t) := x_0 + \int_{t_0}^{t} V_\epsilon[\rho]\bigl(s, X_{\rho,\epsilon}[t_0,x_0](s)\bigr) \, ds, \qquad t \in [t_0,T],
\end{align}
for $(t_0,x_0) \in [0,T] \times \mathbb{R}$, or equivalently the ODE
\begin{align*}
    \frac{d}{dt} X_{\rho,\epsilon}[t_0,x_0](t) &= V_\epsilon[\rho]\bigl(t, X_{\rho,\epsilon}[t_0,x_0](t)\bigr), \qquad t \in (t_0,T], \\
    X_{\rho,\epsilon}[t_0,x_0](t_0) &= x_0. \nonumber
\end{align*}
The deterministic characteristics $\bar{X}$ for \eqref{eq:EsNV} are defined similarly, with $V_\epsilon$ replaced by $\bar{V}$.
\end{definition}
\noindent
For a discussion of the well‑posedness and the explicit Euler time‑marching used for the numerical sampling, see Chapter 3.5 of \cite{Boehme_2025} and the references therein.

\begin{example}\label{ex:standart_ex}
As accompanying examples for the following sections, we consider two different initial traffic conditions. We either assume a low congestion case $\rho_0^{\text{low}}$ or a high congestion case $\rho_0^{\text{high}}$. Thus, we consider the initial data for $x \in \mathbb{R}$:\\
\noindent%
\begin{minipage}{0.5\textwidth}
\centering
\[
\rho_0^{\text{low}}(x) =
\begin{cases}
0.5, & \text{if } x\in\left[ \frac{1}{3}, \frac{2}{3} \right], \\
0.2, & \text{else},
\end{cases}
\]
\end{minipage}%
\begin{minipage}{0.5\textwidth}
\centering
\[
\rho_0^{\text{high}}(x) =
\begin{cases}
0.9, & \text{if } x\in\left[0, 2 \right], \\
0.2, & \text{else}.
\end{cases}
\]
\end{minipage}\\
\noindent
For the rest of our work, we further fix the base-velocity $v(\rho)=1-\rho^2$ with $\rho^{\max}=1$ and the concave kernel 
$$
W_\eta^{\text{conc.}}(x) = \frac{3}{2\eta^3}\left(\eta^2-x^2\right), \quad \text{with} \quad \eta = 0.2.
$$
\end{example}
\noindent
Using Example \ref{ex:standart_ex}, we compare the characteristics of 15 \eqref{eq:sNV} realizations to the ones of \eqref{eq:EsNV} in Figure \ref{fig:chars_rho_high}. The noise is generated using the Markovian process \eqref{eq:JP_resampling}, as introduced in the following section.

\begin{figure}[htb!]
    \centering
    \includegraphics[width=1\textwidth]{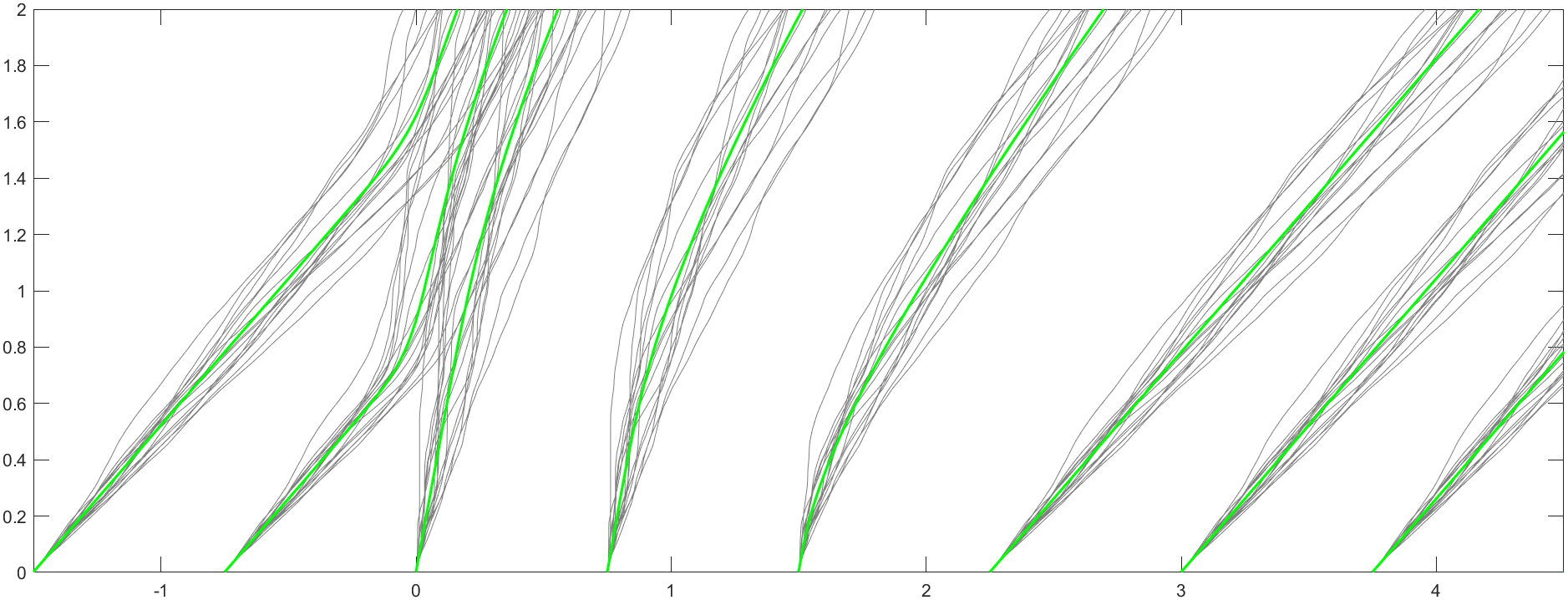}
    \caption{Characteristics based on $\rho_0^{\text{high}}$ to \eqref{eq:sNV} in grey and
    \eqref{eq:EsNV} in green.}
    \label{fig:chars_rho_high}
\end{figure}

Note that for any fixed realization $\omega$, the characteristics of \eqref{eq:sNV} do not cross. As we can see, the deterministic characteristics of the expected flux model \eqref{eq:EsNV}  provide a good fit to the average propagation of the stochastic characteristics, despite the natural increase in variance due to noise accumulation.
In Section \ref{sec:num_results} an error analysis in the characteristics-space using the same high-noise-autocorrelated error processes will be carried out. The latter is now introduced. 

\subsection{Markov-type error processes}\label{subsec:markov_type_error}
Compared to the initial white‑noise approach of \cite{Boehme_2025}, we build on the relaxed assumption (C), which now allows for auto‑correlated noise increments in \eqref{eq:err_very_general}. As our numerical evaluations in Section \ref{sec:num_results} will demonstrate, the use of such increments can generate substantially stronger fluctuations in the resulting densities while also supporting the physical interpretation.\\
Markov chains are natural candidates for  our error processes. A suitable definition of Markov chains for our purposes as a dynamical system can be extracted from Proposition 11.6 in \cite{kallenberg2021foundations} and page 10 of \cite{eberlein}, which states that any Markov process $(X_k)_{k \in \mathbb{N}_0}$ with values on $(S, \mathbb{B}(S))$ where $S \in \mathcal{B}(\mathbb{R}),$ can be written as a dynamical system
$$ X_{k+1}=f_{k}(X_k,U_{k+1}), \qquad k \in \mathbb{N}_0,$$
where the maps $f_k:S \times S' \rightarrow S$ are measurable and $X_0,U_1, U_2, \ldots, U_k, \ldots$ are independent random variables
with $U_k$ taking values in a measurable space $(S', \mathcal{S}')$, allowing for an  acceptance-rejection sampling, which we leverage below. 
Using this definition, the transition probabilities of the Markov chain satisfy
$$ \mathbb{P}(X_{k+1} \in B | X_k=x)= \mathbb{P}( f_k(x,U_{k+1}) \in B) $$
for any $x \in S$, $B \in \mathcal{B}(S)$.
Naturally, we will work with $S=[-\tau,\tau]$ in the following.
\begin{definition}[Admissible Markovian error process]
\label{def:markov_error}
Let $(X_k)_{k \in \mathbb{N}_0}$ be a Markov chain with values in $(S, \mathbb{B}(S))$ and $X_0=0$. Then, the process
\begin{align*}
    \epsilon(t)=\sum_{k=0}^{R_T} X_k \chi_{[t^k,t^{k+1})}(t), \qquad t\in[0,T],
\end{align*}
where $t^k = k\delta_R$ and $R_T = \lfloor T/\delta_R \rfloor < \infty$, is called an admissible Markovian error process. 
\end{definition}
\noindent
The white noise of \cite{Boehme_2025} reads in the above setting  as
\begin{align}\label{eq:WN_sampling}
    X_{k+1}=f_k(X_k,U_{k+1})=U_{k+1}, \qquad X_0=0,\;\; U_1, \ldots, U_{k+1} \sim \mathcal{U}((-\tau,\tau)),
\end{align}
which directly gives the time-independent expression of the expected velocity, based on the corresponding admissible Markovian error process and definition of $v_a$:
\begin{align}\label{eq:bar_v_WN}
    \bar{v}(\bar{m},t) = \int_{-\tau}^{\tau} v_a(\bar{m},t)\nu_t(da)=\frac{1}{4\tau}\Bigl( \bigl(\tau+v(\bar{m})\bigr)^2 -\max\{0,v(\bar{m})-\tau\}^2 \Bigr).
\end{align}
Compared to this, the Markovian increments allow for modeling temporal correlation structures, e.g., by using a (particular) Jacobi-process.
So consider the stochastic differential equation (SDE)
\begin{align}\label{eq:JP_SDE}\tag{JP}
d\epsilon(t) = -\alpha\,\epsilon(t)\,dt + \sigma\sqrt{\bigl(\epsilon(t)+\tau)(\tau-\epsilon(t)\bigr)}\,dW_t, \quad t \in [0,T],  \quad \epsilon(0) = 0,
\end{align}
with a standard Wiener process $W=(W_t)_{t \in [0,T]}$, a symmetric bound $0<\tau\leq v^{\max}$ and noise parameters $\alpha,\sigma>0$, controlling the mean reversion and volatility. This process is a symmetric version of the classical Wright-Fisher diffusion process as considered, for example, in \cite{JP_finance} or  \cite{JP_general}. 
We have
$$ \mathbb{E}[\epsilon(t)]=0, \qquad t \geq 0$$
 and
 $$  \mathbb{E}[\epsilon(t)^2]= \frac{\sigma^{2}\tau^2}{2 \alpha + \sigma^2} \bigl(1- \exp(- (2\alpha+\sigma^2) t)\bigr), \qquad t \geq 0, $$
See, e.g., \cite{JP_finance}, also for expressions of higher moments. Moreover, the sample paths of the Jacobi process remain in $[-\tau,\tau]$, that is,
\begin{align}\label{eq:JP_in_tau}
    \mathbb {P}\bigl(\epsilon(t) \in [-\tau,\tau], \, t \in [0,T]\bigr)=1. 
\end{align}
We emphasize that uniform boundedness is a key property of the Jacobi process, which is generally not given for classical diffusion models such as the (generalized) Ornstein–Uhlenbeck process.
Although exact simulation of the Jacobi process is feasible, see e.g.,\ \cite{On_Wright_Fisher_basic_simulation}, it is computationally quite expensive. 
However, one can use an acceptance-rejection Euler scheme for its simulation on a grid 
$0=t_0<t_1<\ldots<t_N=T$. 
Set $X_0=0$ and 
\begin{align}\label{eq:JP_resampling}
    X_{k+1}= \Psi\!\left( 
     X_k(1-\alpha (t_{k+1}-t_k)) + \sigma \sqrt{(X_k+\tau)(\tau-X_k)}(W_{t_{k+1}}-W_{t_k})
    \right)
\end{align}
for $k=0,\ldots,N-1$,
where $\Psi$ denotes the acceptance-rejection operator, defined for a random variable $Z$ by
\begin{align*}
\Psi(Z) =
    \begin{cases}
    Z, & |Z|\le \tau,\\
    \Psi(Z'), & |Z|>\tau,
    \end{cases}
\end{align*}
with $Z'$ an independent copy of $Z$. 
This correction is only required due to the numerical grid $t_{k+1}-t_k>0$ as the continuous Jacobi process itself remains in $[-\tau,\tau]$ in the sense of \eqref{eq:JP_in_tau}.
\begin{remark}
In contrast to a projection approach as, e.g., 
\begin{align*}
    \Pi(z)= \max \{ -\tau, \min \{ z, \tau \} \}, \qquad z \in \mathbb{R},
\end{align*}
the resampling scheme does not create point masses the boundaries.
\end{remark}

\begin{remark}
    Another benefit of the Jacobi process is that its distribution $\nu_t$ can be characterized through the evolution of its density $f_t^\epsilon(a)$, which follows a Fokker-Planck evolution (see e.g., \cite{Risken1996}) which is the solution to the PDE:
\begin{align}\label{eq:FP}\tag{FP}
    \partial_t f^{\epsilon}_{t}(a) 
    = \alpha \partial_a (a f^{\epsilon}_{t}(a)) + \frac{1}{2} \sigma^2 \partial_{aa}^2 \Bigl( (\tau^2 - a^2) f^{\epsilon}_{t}(a) \Bigr), \quad f_0^\epsilon(a)=1\chi_{\{0\}}(a).
\end{align}
Hence, $\bar{v}$ can be explicitly derived as
\begin{align}\label{eq:JP_barv_precise}
    \bar{v}(\bar{m},t) = \int_{-\tau}^{\tau} v_a(\bar{m},t)   f^{\epsilon}_{t}(a)\ da.
\end{align}
However, as we outline in Section \ref{subsec:err_sampl} this remains a theoretical property, as our implementation needs to account not only for the acceptance-rejection sampling of the process itself but also for the piecewise constant error structure of $\epsilon(t)$ as of Assumption (C).
\end{remark}

\section{Numerical discretization}\label{sec:numerics}
For all subsequent evaluations, we assume a time mesh given by $t^n = n\Delta t$ for $n = 0, \dots, N_T$, with $N_T := \left\lfloor T / \Delta t \right\rfloor$. 
Moreover, we make the simplification that the time-dependent error  \eqref{eq:err_very_general} evolves on a finer (generally unknown) $\delta_R$-grid than our numerical observations. 
That is, $\delta_R \leq \Delta t$. This allows us to neglect additional correlations between consecutive observations introduced by the piecewise constant nature of the error. 
For further rationale, we refer to Section~3.1 of \cite{Boehme_2025}, but note that this is equivalent to assuming that decreasing $\Delta t$ improves not only the approximation of the conservation law, but also the representation of the error term until $\Delta t < \delta_R $ is obtained. However, this is purely of a technical nature.

\subsection{Noise sampling and evaluation}\label{subsec:err_sampl} 
Since we require only $N_T$ evaluations of $\epsilon(t)$ (as in Definition \ref{def:markov_error}), we can directly sample from the generating Markov chain by setting the time increments equal:
\begin{align}\label{eq:epsilon_sample}
    \epsilon^n := \epsilon(t^n) = X_{k_n}, \quad \text{where } k_n = \big\lfloor n\Delta t / \delta_R \big\rfloor.
\end{align}
For simplicity, we take $\delta_R = \Delta t$ in our numerical schemes, so that we can ease the notation by setting $k_n = n$.
If we require a finer noise grid as, for example, in convergence analysis of $\Delta t$, we can always generate the error process with $\delta_R < \Delta t$ and then evaluate it at the coarse grid points.\\
In the white-noise case, the sampling of $\epsilon(t)$ is straightforward by drawing at least $N_T$ i.i.d.\ realizations according to \eqref{eq:WN_sampling}, thus constructing an admissible (though time-independent) process as in Definition~\ref{def:markov_error}. 
In addition, $\bar{v}$ is given by \eqref{eq:bar_v_WN}.\\

However, using the Jacobi process as a generator requires a more careful implementation. We have already established that the increments of the admissible Markovian error process are drawn using the acceptance‑rejection approach from \eqref{eq:JP_resampling}.
Hence, we are left with just identifying 
\begin{align*}
    t_{n+1}-t_n = \Delta t \quad \text{and}\quad W_{t_{n+1}}-W_{t_{n}} \sim \mathcal{N}(0, \Delta t).
\end{align*}
\noindent
To compute $\bar{v}(\bar{m},t)$ as in \eqref{eq:eyp_density}, we need access to the distribution $\nu_t$ of $\epsilon(t)$. Although the continuous process satisfies \eqref{eq:FP}, the construction of $\epsilon(t)$ requires a separate approximation of its density at every time step.
To do so, we fix a simulation grid of the probability space with cell centers $\mathcal{A}=\{a_j\}_{j=1}^M \subset [-\tau, \tau]$ and integration weights $\{w_j\}_{j=1}^M$, corresponding to the represented cell-width.
We denote by $\hat{f}_n$ the numerical density of $\epsilon(t^n)$ and set the initial density $\hat{f}_0$ to approximate the cell centered around zero with mass one.
From there, we evolve $\hat{f}_n$ on the same time grid as before using forward equations with respect to the update scheme \eqref{eq:JP_resampling}.
More precisely, given any realization $\epsilon(t^n)=a_j$, the Markov property together with the Euler-Maruyama implementation allows us to assume a conditionally Gaussian distribution for the proposal step of \eqref{eq:JP_resampling}.
Thus, for each possible grid-realization $a_j \in \mathcal{A}$, we compute 
\begin{align*}
\mu_j &= a_j(1-\alpha\Delta t), \qquad
\sigma_j^2 = \sigma^2(\tau^2-a_j^2)\Delta t,
\end{align*}
allowing us to calculate the density after the acceptance-rejection update as
\begin{align*}
\hat{f}_{n+1}(a_j) = \sum_{i=1}^{M} \hat{f}_n(a_i) \, \frac{\phi(a_j; \mu_i, \sigma_i^2)}{\sum_{k=1}^{M} \phi(a_k; \mu_i, \sigma_i^2) w_k} \, w_i, \qquad a_j \in \mathcal{A},
\end{align*}
where $\phi(\cdot; \mu_i, \sigma_i^2)$ denotes a standard Gaussian Kernel. 
Note, that the denominator is necessary to redistribute the mass of the unbounded kernel to $[-\tau,\tau]$, thus respecting the acceptance-rejection approach as of \eqref{eq:JP_resampling}.
This then allows us to calculate $\bar{v}(\bar{m},t)$ via quadrature of \eqref{eq:JP_barv_precise}, such that:
\begin{align}\label{eq:JP_barv_numeric}
\bar{v}(\rho, t^n) \approx \sum_{j=1}^{M} \max\!\bigl(0, v(\rho) + a_j\bigr) \hat{f}_n(a_j) w_j.
\end{align}
Finally, we apply the Godunov scheme for \eqref{eq:EsNV}, which we outline after two short remarks.
\begin{remark}
\leavevmode
    \begin{itemize}[leftmargin=1.2em, itemsep=1pt, topsep=1pt]
        \item Alternatively, one can employ Monte Carlo sampling combined with kernel density estimation which, however, introduces additional hard-to-control noise. 
        \item We use a Chebyshev grid $\mathcal{A}$, which clusters points near the boundaries $\pm\tau$, since $\hat{f}_n$ can exhibit rapid gradients there. Although a rigorous analysis of the approximation error is outside the scope of this work, we found that $M=601$ points provide density profiles that match those obtained from high-sample Monte Carlo comparisons, while keeping the runtime short through parallel computation, see Figure~\ref{fig:density_over_time} for an example.
    \end{itemize}
   \begin{figure}[htb]
    \centering
    \includegraphics[width=0.5\linewidth]{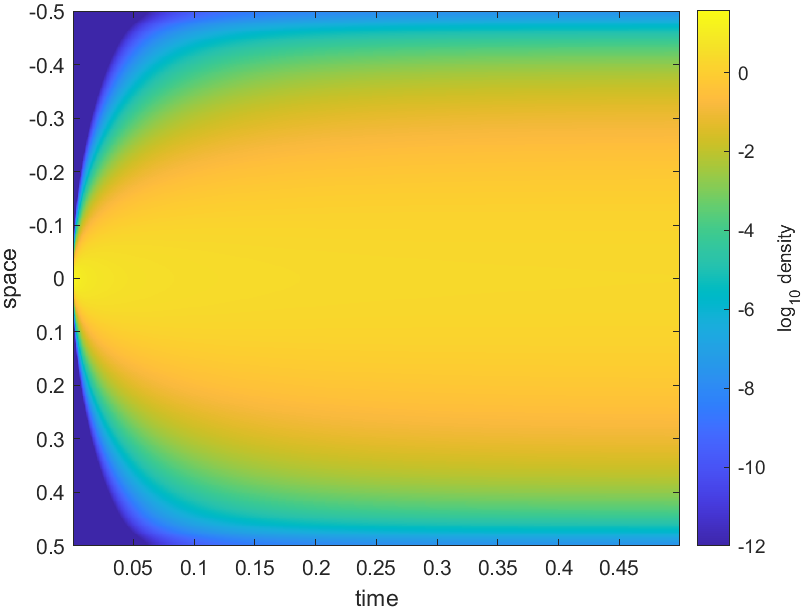}
    \caption{Evolution of $\hat{f}_n$ ($\alpha=4, \sigma=1, \tau=0.5$) on a 601-node Chebyshev grid}%
\label{fig:density_over_time}
\end{figure} 
\end{remark}

\subsection{Numerical scheme for the sNV model}\label{subsec:snv_scheme} 
For an in-depth description of the numerical scheme alongside convergence proofs, we refer to \cite{Boehme_2025}, but repeat its central points, while slightly expanding the notation with respect to $\epsilon(t)$ sampled as described before.
As typical for Godunov type schemes, we assume an equidistant spatial grid with cell centers $x_j$, cell interfaces $x_{j-1/2}$ and cell length $\Delta x= x_{j+1/2}-x_{j-1/2} \; \forall j \in \mathbb{Z}$. 
Let, as usual, $\rho_j^n:=\rho(t^n,x_j)$ and define the piecewise constant function 
\begin{align}\label{eq:mesh}
	\rho^{\Delta x}(t,x)=\rho_j^n \text{ for } (t,x)\in[t^n,t^{n+1})\times[x_{j-1/2},x_{j+1/2}).
\end{align} 
Then, the initial deterministic density $\rho_0$ is discretized by the cell averages with respect to (\ref{eq:mesh}).
In each time step, the Riemann problems arising at the discontinuities between the numerical densities $\rho^n_j, \; j \in \mathbb{Z}$ are then solved exactly until the first shocks collide. Thus, the update of the cell densities is calculated as
\begin{align*}
    \rho_j^{n+1}=\rho_j^{n}+\frac{\Delta t}{\Delta x}\left(F^n_{j+1/2}(\rho_j^n)-F^n_{j-1/2}(\rho_{j-1}^n)\right),
\end{align*}
where the numerical flux $F^n_{j+1/2}$ is based on the solution to the Riemann problems at the cell interfaces and the actual flux.
Defining
\begin{align*}
    v_{\epsilon}^n(\rho_j^n)&:=v_\epsilon(\rho_j^n,t^n)=\max\{0,v(\rho_j^n)+\epsilon^n\},
\end{align*}
where $\epsilon^n$ is sampled as in \eqref{eq:epsilon_sample} according to its generating Markov chain, which may be white noise \eqref{eq:WN_sampling}, a Jacobi process \eqref{eq:JP_resampling} or any other admissible generator. We employ the stochastic upwind flux
\begin{align*}
	F_{j+1/2}^n(\rho_j^n)=\rho_j^n V_{\epsilon,j}^n, \;\; \text{with} \;\; V_{\epsilon,j}^n &= \sum_{k=0}^{N_\eta-1} \gamma_k v_\epsilon^n(\rho^n_{j+k+1}), \quad N_\eta=\left\lfloor \eta/\Delta x\right\rfloor,
\end{align*}
given the kernel evaluation
\begin{align*}
	\gamma_k &=\int_{k \Delta x}^{(k+1) \Delta x} W_\eta(x) \,dx, \quad k=0,\dots,N_\eta-1. 
\end{align*}
Thus, our stochastic time step update reads 
\begin{align}\label{eq:SCHEM_sNV}
	\rho_j^{n+1}=\rho_j^{n}-\lambda \left(\rho_j^n V_{\epsilon,j}^n-\rho_{j-1}^n V_{\epsilon,j-1}^n\right), \quad \lambda:=\frac{\Delta t}{\Delta x}.
\end{align}
As described in \cite[Rem. 4.1]{Boehme_2025}, the stochastic CFL condition
for \eqref{eq:sNV}
\begin{align}\label{eq:sCFL} 
    \lambda \leq \frac{1}{\gamma_0 \lVert v' \rVert \rho^{\max}+\lVert v_\epsilon \rVert(\omega)} := C_{\text{CFL}}^{\max}, 
\end{align}
can be deterministically bounded with the help of the error terms bound $\tau$ as
\begin{align}\label{eq:sCFL_det} 
    C_{\text{CFL}}^{\text{det}}:=\frac{1}{\gamma_0 \lVert v' \rVert \rho^{\max}+v^{\max} + \tau} 
    \leq C_{\text{CFL}}^{\max}.
\end{align}

\subsection{Numerical scheme for the EsNV model}\label{subsec:esnv_scheme}
The numerical scheme for our expectation model \eqref{eq:EsNV} follows directly from the scheme for the sNV model, by replacing the stochastic velocity $v_\epsilon^n$ with the expected deterministic velocity $\bar{v}^n(\cdot) := \bar{v}(\cdot, t^n)$, calculated by
\eqref{eq:bar_v_WN} for white noise or \eqref{eq:JP_barv_numeric} for the Jacobi process.
The corresponding numerical flux is
\begin{align*}
	\bar{F}_{j+1/2}^n(\bar{m}_j^n)=\bar{m}^n \bar{V}_{j}^n, \;\; \text{with} \;\; \bar{V}_{j}^n &= \sum_{k=0}^{N_\eta-1} \gamma_k \bar{v}^n(\bar{m}^n_{j+k+1}), \quad N_\eta=\left\lfloor \eta/\Delta x\right\rfloor,
\end{align*}
where $\bar{m}_j^n$ approximates $\bar{m}(t^n, x_j)$ with $\bar{m}_j^0=\rho_0(x_j)$ discretized as before.
The update is then
\begin{align*}
\bar{m}_j^{n+1} = \bar{m}_j^n - \lambda \left( \bar{m}_j^n \bar{V}_j^n - \bar{m}_{j-1}^n \bar{V}_{j-1}^n \right), \quad \lambda=\frac{\Delta t}{\Delta x}.
\end{align*}
 The CFL condition coincides with the deterministic bound in \eqref{eq:sCFL}, i.e., $\lambda \leq C_{\text{CFL}}^{\text{det}}$.

\section{Numerical results}\label{sec:num_results}
We now focus on the numerical investigation of the two primary extensions of the initial model of \cite{Boehme_2025}: the inclusion of autocorrelated Markovian noise and the derived simplified mean-value PDE \eqref{eq:EsNV}.
The aim of this part is twofold. 
First, in Section \ref{subsec:markovian_results}, we demonstrate that autocorrelated noise structures induce significantly stronger perturbations in the traffic density compared to the white-noise approach,
underscoring the relevance of the now-available Markovian setting. 
Second, motivated by these strong perturbations and the need for a mean-value estimation, Section \ref{subsec:esnv_results} evaluates the accuracy of \eqref{eq:EsNV} as a computational proxy for the latter. Its ability to capture the expectation of the stochastic solutions is assessed through extensive comparisons with Monte Carlo averages on characteristic space.

\subsection{Impact of autocorrelated Markovian noise}\label{subsec:markovian_results} 
We proceed to analyze the impact of autocorrelated Markovian noise exemplarily based on the Jacobi process, sampled according to Section \ref{subsec:err_sampl} with the stochastic Godunov scheme \eqref{eq:SCHEM_sNV} and compare it to the white noise approach of \eqref{eq:WN_sampling}. 
For this purpose, we choose the initial density $\rho_0^{\text{low}}$ and the kernel $W_\eta^{\text{conc.}}$ from Example \ref{ex:standart_ex}, with look-ahead distance $\eta=0.2$. Further, we set $\tau=v^{\max}/2=0.5$ for both error types. Additionally, for the Jacobi process, we fix the mean reversion- and volatility parameter to $\alpha=4$ and $\sigma=1$, respectively. To distinguish the different noise types involved, we add the suffix \texttt{\_JP} and \texttt{\_WN}, for the Jacobi-type and white noise, respectively.\\
\begin{figure}[htb]
    \centering
    \begin{subfigure}{0.5\textwidth}
        \centering
        \includegraphics[width=1.0\linewidth]{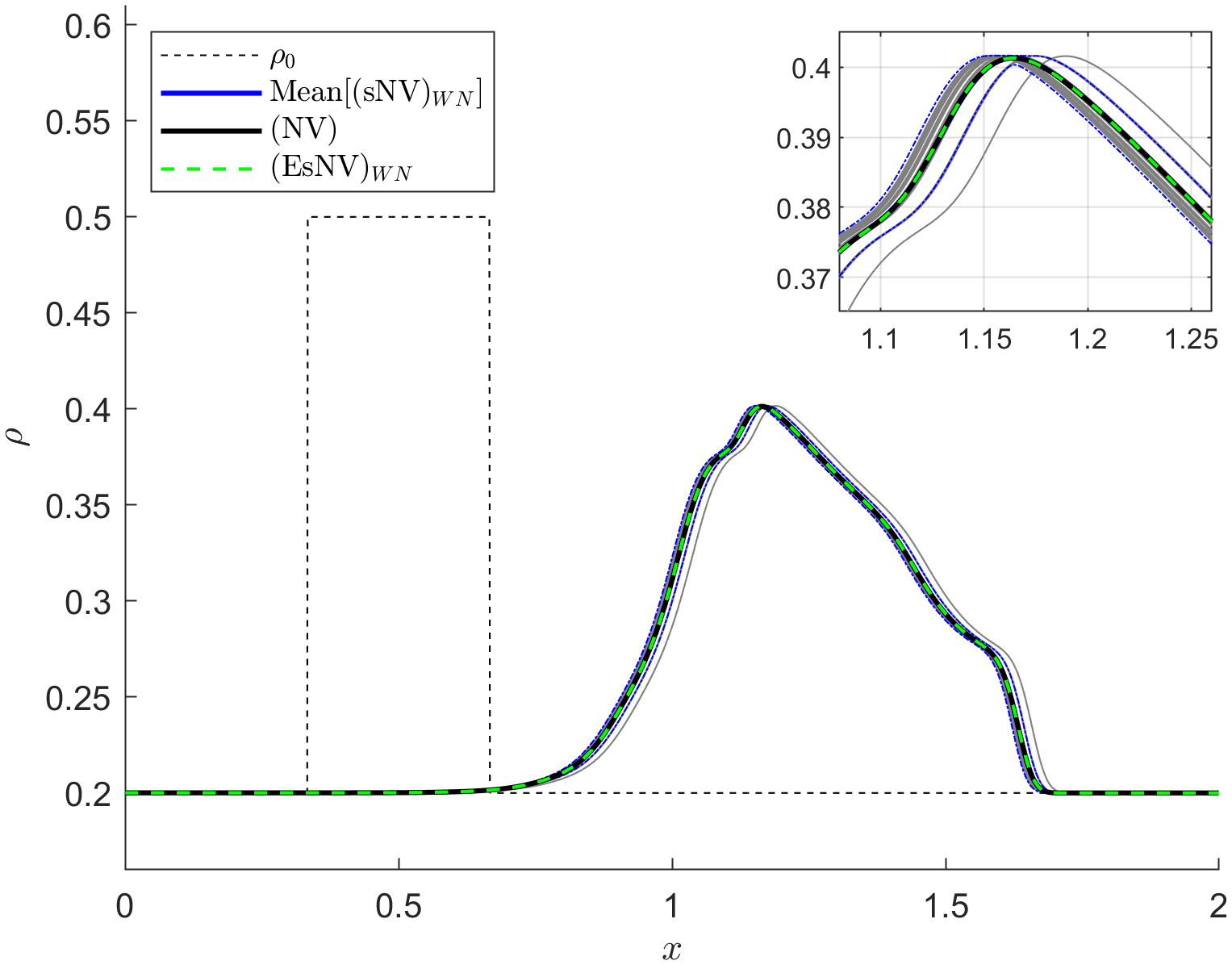}
      \end{subfigure}%
      \begin{subfigure}{0.5\textwidth}
        \centering
        \includegraphics[width=1.0\linewidth]{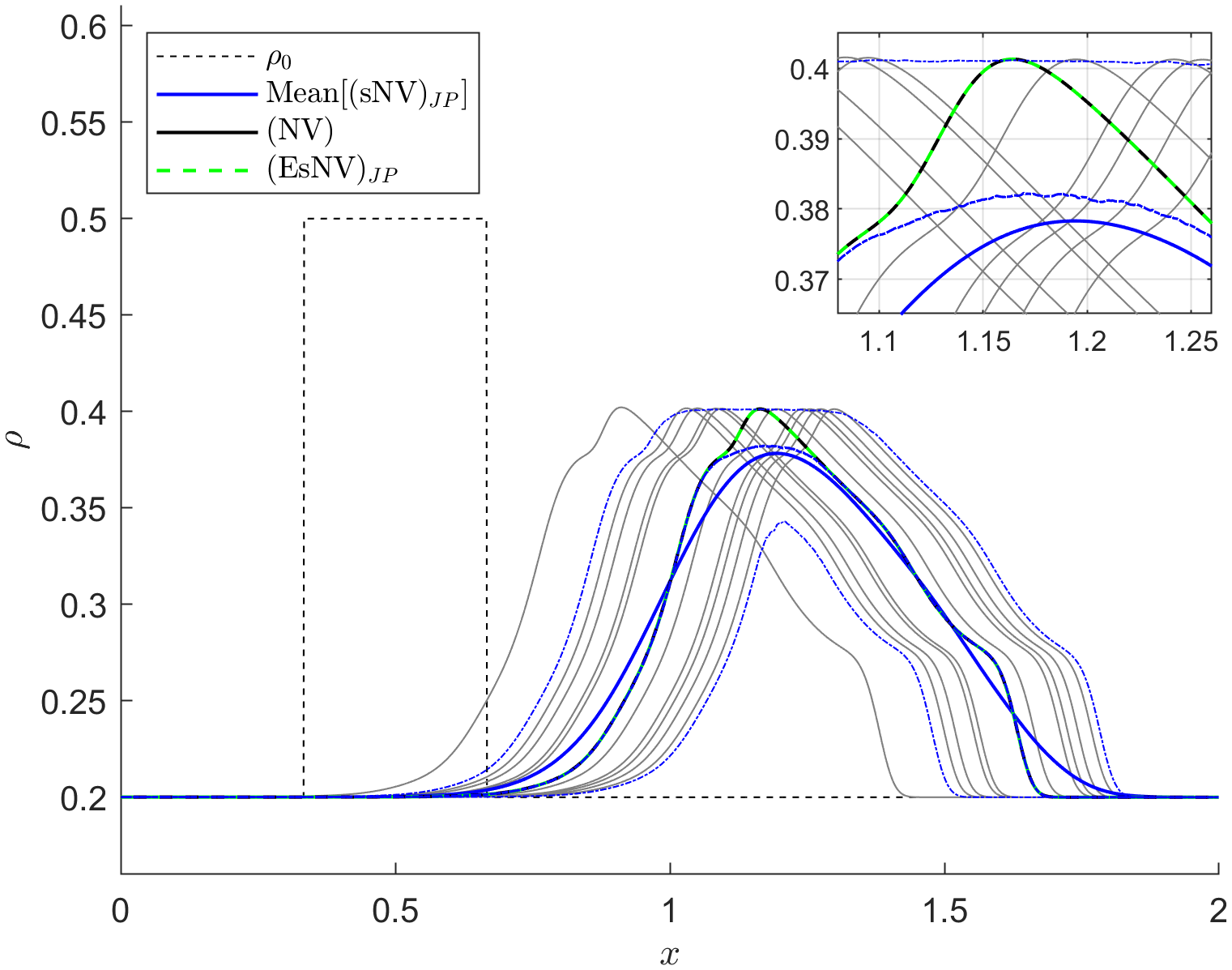}
      \end{subfigure}
    \caption{Realizations (gray), Monte Carlo average and pointwise $\{5, 50, 95\}\%$-quantiles (blue, $M=2\cdot 10^3$) of \eqref{eq:sNV}, including mean-proxy \eqref{eq:EsNV} (green) at $T=1$. Left: white noise. Right: Jacobi-type noise. $\Delta x=10^{-3}$, $\Delta t$ acc. to \eqref{eq:sCFL_det} }
\label{fig:noise_comparison}
\end{figure}

Remarking that the parameters $\alpha$ and $\sigma$ are moderate - with more extreme cases possible - we make the following observations based on Figure \ref{fig:noise_comparison}:
\begin{itemize}[noitemsep]
    \item The Jacobi-type noise produces significantly stronger fluctuations compared to the white noise approach, where each realization represents a valid solution to a differently perturbed \eqref{eq:sNV} system.
    \item In particular, white noise induces only minimal effective perturbation. This is explained by the intrinsic averaging of noise terms within each observation of the process, whereas the temporal correlations introduced by the Markovian structure prevent this intrinsic cancellation.
    \item Consistent with the analysis in \cite{Boehme_2025}, the low initial density causes  \eqref{eq:EsNV} to coincide with the deterministic solution \eqref{eq:NV}. This results from a symmetric error influence and the non-linear max-operator remaining inactive. A counter-example necessitating the calculation of $\bar{v}$ via \eqref{eq:JP_barv_numeric} follows in Section \ref{subsec:esnv_results}.
\end{itemize}
Crucially, the introduced high variability of solutions renders the following pressing:
\begin{itemize}[noitemsep]
    \item Although the Monte Carlo average exhibits numerical diffusion in both cases (a discrepancy that might be negligible for white noise) its prominence for the Jacobi-type noise underlines the need for a proxy describing the expectations, which still obeys a conservation law.
    While the Monte Carlo average may be physically interpreted as an average road utilization over multiple observations, this averaged profile does not represent a valid solution to the dynamics \eqref{eq:sNV} for any single noise realization.
\end{itemize}
The latter finding is exactly the motivation for our derivation of \eqref{eq:EsNV} in Section \ref{subsec:mean_velo}. 
By incorporating the expected stochastic influence directly into the flux function, \eqref{eq:EsNV} remains a valid conservation law, propagating a single expected density at each time step $\Delta t$. However, how this coincides with the analytically unknown expectation $\mathbb{E}[\rho]$ is what we will study now using Monte Carlo averages to approximate the latter.

\subsection{Validation of (EsNV) and benefit of nonlocality}\label{subsec:esnv_results}
For the following, we utilize the same parameters as before, but employ $\rho_0^{\text{high}}$ from Example \ref{ex:standart_ex}. We choose this setting because it presents a more challenging scenario compared to $\rho_0^{\text{low}}$. In particular \eqref{eq:EsNV} no longer coincides with \eqref{eq:NV}. Instead, the calculation now requires the stochastic density to capture the portion of the distribution affected by the active max-operator, which  necessitates the calculation as of Section~\ref{subsec:err_sampl}.
Furthermore, given the particular need for a suitable approximation in the context of the autocorrelated noise, we will focus our analysis on the latter.

\begin{figure}[htb]
    \centering
    \begin{subfigure}{.5\textwidth}
        \centering
        \includegraphics[width=1.0\linewidth]{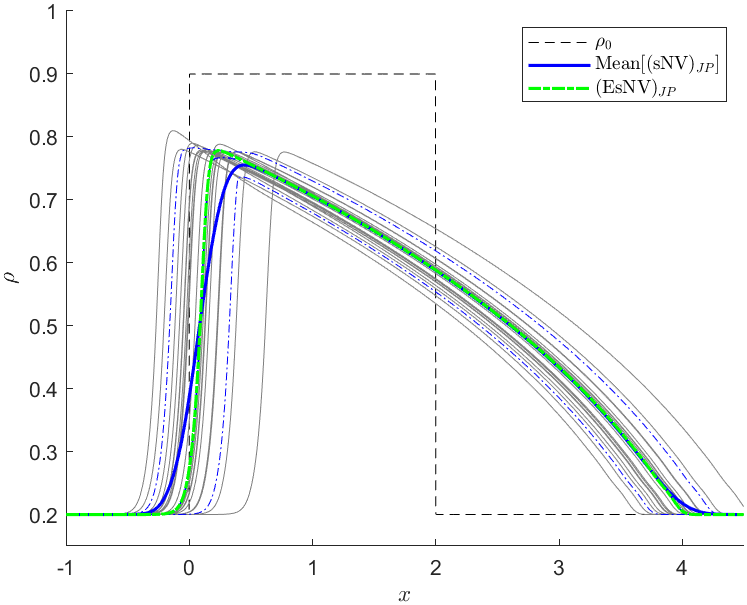}
      \end{subfigure}%
      \begin{subfigure}{.5\textwidth}
        \centering
        \includegraphics[width=1.0\linewidth]{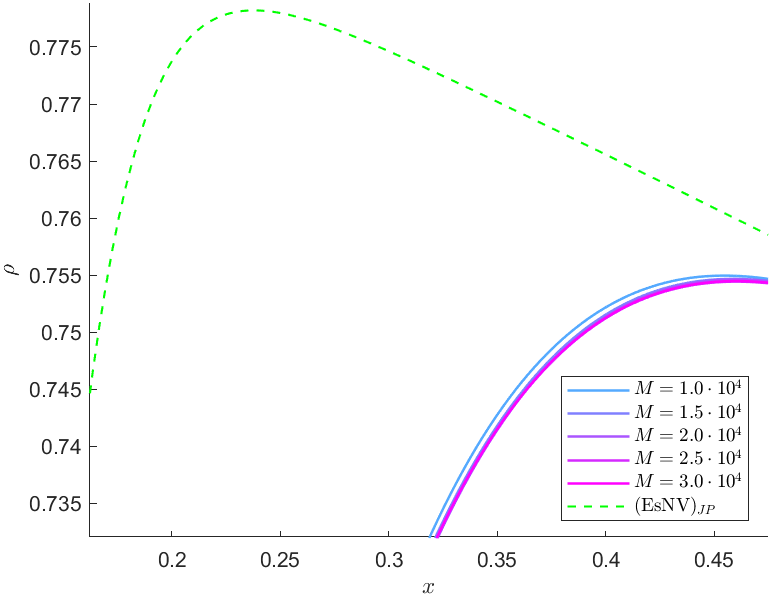}
      \end{subfigure}
      \caption{Left: Realizations (gray), Monte Carlo average and pointwise $\{5, 50, 95\}\%$-quantiles (blue, $M=2\cdot 10^3$) of \eqref{eq:sNV}, including mean-proxy \eqref{eq:EsNV} (green) at $T=2$ for Jacobi-type noise. Right: Zoom in on multiple averages for increasing numbers of samples.
      $\Delta x=3\cdot10^{-3}$, $\Delta t$ acc. to \eqref{eq:sCFL_det}}
\label{fig:noise_comparison_high}
\end{figure}

Starting at the density level, we observe that in regions of smooth propagation, e.g., $x\in[1,3]$, the average coincides well with \eqref{eq:EsNV}. However, near rapid changes, e.g., around $x=0$, the average inevitably smooths out the shocks.
Crucially, increasing the sample size does not reduce this effect, as the average of random jumps results in a linear interpolation. Verifying this on the right of Figure \ref{fig:noise_comparison_high}, we conclude that the Monte Carlo average of the i.i.d densities $\rho^{\Delta x,(k)}$ may provide insight into expected road usage, but does not - or at least not fast enough - converge to a valid solution of \eqref{eq:sNV} for a specific realization of the noise or to \eqref{eq:EsNV}. \\
However, \eqref{eq:EsNV} lies close to the pointwise median and is generally centered among the realizations. Additionally, as seen in Figure \ref{fig:chars_rho_high}, this relation is mirrored in characteristic space.
To validate these qualitative observations, we circumvent the smoothing of the density averaging by shifting the analysis to the characteristic space.
By averaging the realized characteristics given by \eqref{eq:chars_integral_eq} instead of the densities, we exploit mass conservation and make use of the trajectories themselves, which are more robust to shocks. In doing so, we define the average characteristic - or equivalently the average position of a particle - as:

\begin{definition}[Characteristic Monte Carlo Average]\label{def:MCM_char}
    For given starting values $t_0$ and $x_0$, define the average of realized characteristics as:
    \begin{align*}
         \bar{X}_\rho^M[t_0,x_0](t) := \frac{1}{M} \sum _{k=1}^M X^{(k)}_{\rho,\epsilon}[t_0,x_0](t) \approx\mathbb{E}\bigl[X_{\rho}[t_0,x_0](t)\bigr],
    \end{align*}
    where the superscript $(k)$ denotes the i.i.d. copies of the solutions to the stochastic integral equation \eqref{eq:chars_integral_eq}.
\end{definition}
\noindent
As depicted in Figure~\ref{fig:strong_dev_char}, the linear Monte Carlo averaging of the characteristics around the shock-front at $x=0$ preserves the forward propagation of particles without introducing smoothing. Hence, the average given by Definition~\ref{def:MCM_char} maintains the conservation dynamics, 
which is effectively captured by \eqref{eq:EsNV} at the characteristic level. By the duality between characteristic and density space, it follows that the associated density is indeed the most likely realization visible in Figure~\ref{fig:chars_rho_high}. 

\begin{figure}[htb]
    \centering
    \includegraphics[width=1\textwidth]{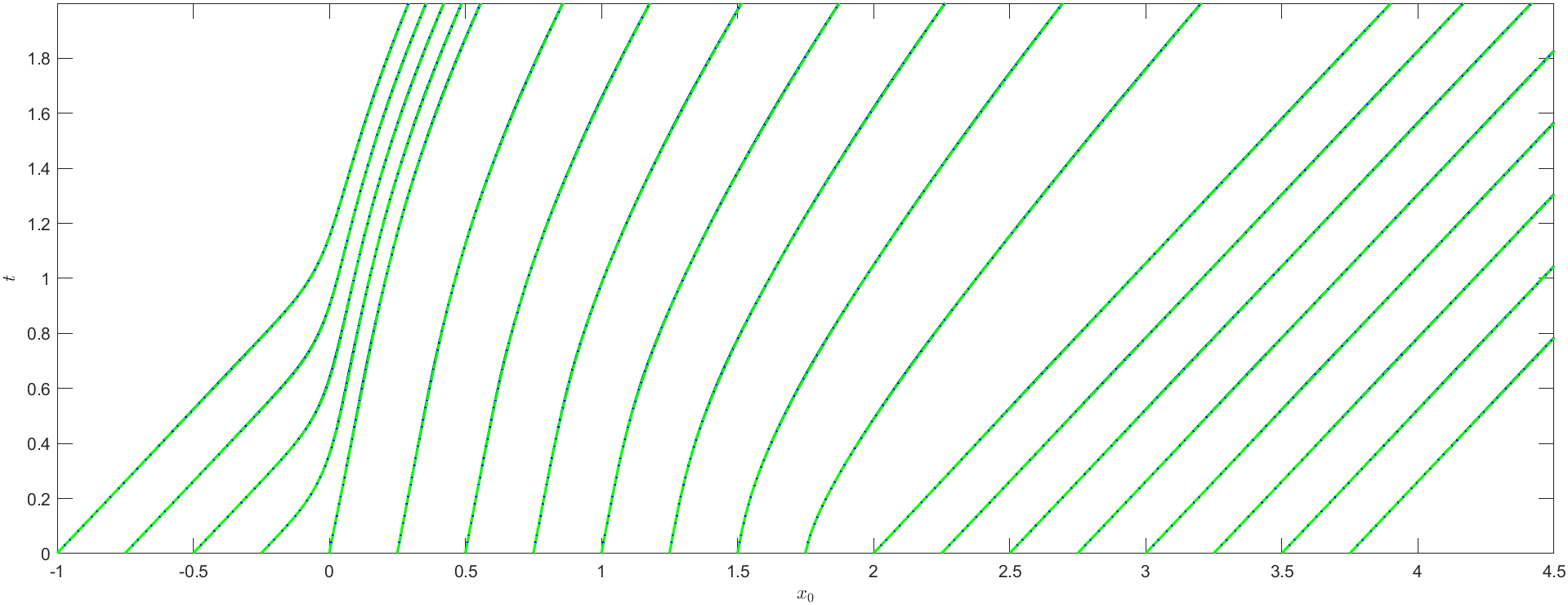}
    \caption{Characteristics of Figure \ref{fig:noise_comparison_high}. The Characteristic Monte Carlo average ($M=2\cdot10^3$) in blue, with \eqref{eq:EsNV} mostly overlaying the average in dashed green lines. $\Delta x=3\cdot10^{-3}$, $\Delta t$ acc. to \eqref{eq:sCFL_det}.}
\label{fig:strong_dev_char}
\end{figure}
\noindent
We substantiate the above by providing numerical evidence that the average as of Definition~\ref{def:MCM_char} coincides with \eqref{eq:EsNV}. If $X_{\bar{m}}$ denotes the characteristics of \eqref{eq:EsNV}, we define the Monte Carlo $\ell^1$-bias in the characteristic space as:
\begin{align*}
    \ell^1[t_0,x_0](M) := \left| \bar{X}_\rho^M - X_{\bar{m}} \right|(t_0,x_0).
\end{align*}
Then, for a fixed set $(t_0,x_0)$, we provide the results of Monte Carlo computations constituting numerical evidence for the following hypotheses:
\begin{enumerate}[itemsep=1em, label=\Roman*.]
    \item As $M \to \infty$, the reduction of the Monte Carlo error leads to a decrease in $\ell^1$.
    \item For a fixed $M$, the approximation error induced by \eqref{eq:exp_approx} is reduced by using finer time steps, i.e., for $\Delta t \to 0$,  $\ell^1$ is reduced. 
    \item Consequently,
     $ \displaystyle  \lim_{M \to \infty} \lim_{\Delta t \to 0} \ell^1 = 0.$
\end{enumerate}
The numerical results are given in Figure \ref{fig:JP_noise_err}, which illustrates the  rates with respect to both parameters.
\begin{figure}[htb]
    \centering
     \includegraphics[width=0.9\textwidth]{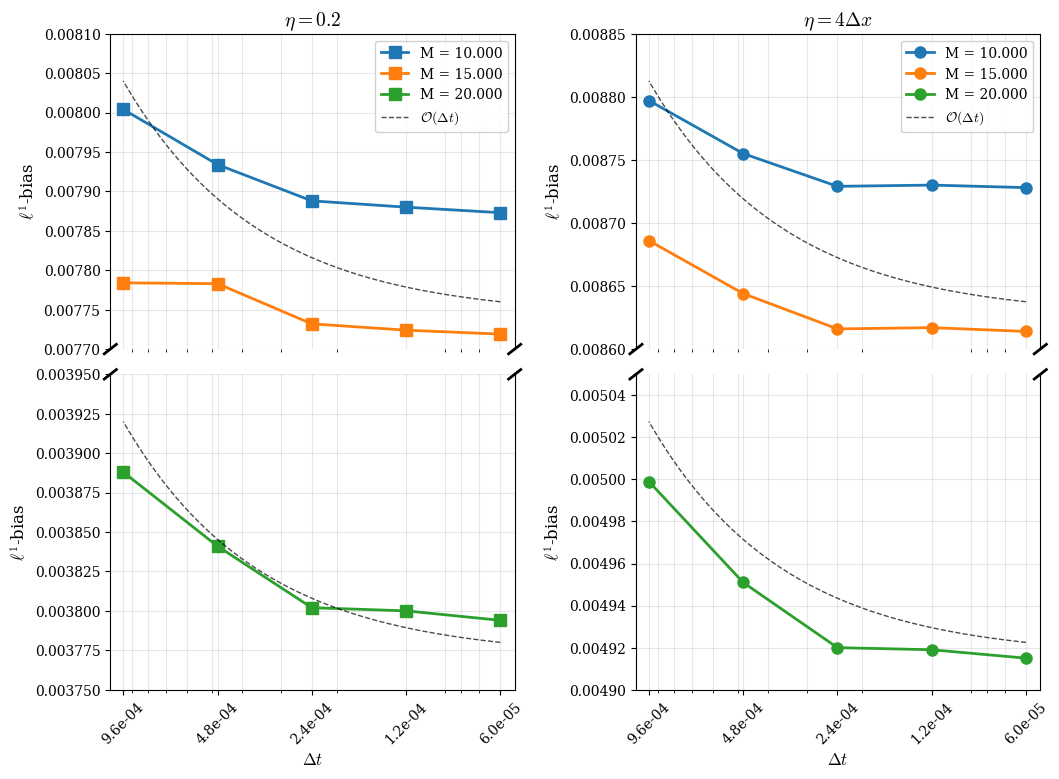}
  \caption{$\ell^1$-bias for decreasing $\Delta t$ and increasing $M$. 
    Setting $t_0=0$ and $x_0$ as in Fig.~\ref{fig:strong_dev_char}, with $\Delta x=3 \cdot 10^{-3}$. 
    The look-ahead distance is set to medium (left) or very short (right). }
    \label{fig:JP_noise_err}
\end{figure}
Additionally, the bias levels are observed to worsen as we approach the local limit (see Figure \ref{fig:JP_noise_err}, right), when deviating from our standard nonlocal setting (left). 
This is attributed to the induced smoother driving behavior and reduced shock propagation of nonlocal models and constitutes one of the main benefits of implementing stochastic models to nonlocal models compared to existing local alternatives.\\
As such, by choosing sufficiently small time steps $\Delta t$, the propagation of the expected density reflects the average particle position, without the necessity of modeling additional correlation structures of the stochastic densities, effectively decoupling the expectation from the full probability measure. 
Thus, through the relation between particle trajectories and the density profile, \eqref{eq:EsNV} does yield an expected density profile that preserves shock structures, while the Monte Carlo average of the density merely represents the average road occupancy (e.g., observations over multiple days) and consequently lacks the physical sharpness of a valid single realization.
In conclusion, the $\ell^1$-bias justify \eqref{eq:EsNV} as an effective expectation-model.

\section{Conclusion and outlook} \label{sec:conc_outlook}
The presented work strengthens the initial framework of the sNV model by consolidating its theoretical foundations and significantly expanding its modeling capabilities, notably through the derivation and validation of a mean-value proxy.
The key contributions of this paper are threefold, building on one another. 
First, we set the theoretical foundation by proving the measurability of the random weak entropy solutions, thereby ensuring the existence of a well-defined expectation. 
Then, we took a modeling perspective and expanded the framework to allow for Markovian noise.
Specifically, we introduced a suitable noise process of Jacobi type that bridges the three-dimensional gap between required regularity, modeling flexibility, and physical interpretation.
Inspired by the stronger perturbations of this approach and the knowledge of the existence of an expectation, we introduced a local solution operator to capture the effect of the noise, leading us to a deterministic mean-value hyperbolic PDE. 
Lastly, we highlighted the smoothing of standard Monte Carlo simulations when faced with low-regularity densities. To address this, we presented an alternative perspective that uses a characteristic Monte Carlo average, which allowed us to validate the use of the proxy model.\\
For future research, we plan to benchmark the modeling capabilities of the sNV model against real-world data. 
In this context, the calibration of the stochastic parameters, the application of the expectation proxy, and the influence of the non-local parameter are of particular interest. Especially, exploring how varying the non-local look-ahead range impacts the formation and propagation of traffic waves under uncertainty will provide deeper insights into the role of stochastic driver anticipation.

\bibliographystyle{siam}
\bibliography{arxiv-bibliography}

\begin{thebibliography}{10}

\bibitem{Abdelmalik2016}
{\sc M.~R.~A. Abdelmalik and E.~H. van Brummelen}, {\em Moment closure approximations of the boltzmann equation based on $\varphi$-divergences}, Journal of Statistical Physics, 164 (2016), pp.~77--104.

\bibitem{JP_general}
{\sc D.~Ackerer, D.~Filipović, and S.~Pulido}, {\em The jacobi stochastic volatility model}, Finance and Stochastics, 22 (2018), p.~667–700.

\bibitem{Barth_2016}
{\sc A.~Barth and F.~G. Fuchs}, {\em Uncertainty quantification for hyperbolic conservation laws with flux coefficients given by spatiotemporal random fields}, SIAM Journal on Scientific Computing, 38 (2016), p.~A2209–A2231.

\bibitem{BlandinGoatin2016}
{\sc S.~Blandin and P.~Goatin}, {\em Well-posedness of a conservation law with non-local flux arising in traffic flow modeling}, Numerische Mathematik, 132 (2016), pp.~217--241.

\bibitem{Boehme_2025}
{\sc {Böhme, T.}, {Göttlich, S.}, and {Neuenkirch, A.}}, {\em A nonlocal traffic flow model with stochastic velocity}, ESAIM: M2AN, 59 (2025), pp.~487--518.

\bibitem{Tosin2021_conference}
{\sc F.~A. Chiarello}, {\em An overview of non-local traffic flow models}, in Mathematical Descriptions of Traffic Flow: Micro, Macro and Kinetic Models, G.~Puppo and A.~Tosin, eds., Cham, 2021, Springer International Publishing, pp.~79--91.

\bibitem{Chiarello2018}
{\sc F.~A. Chiarello and P.~Goatin}, {\em {Global entropy weak solutions for general non-local traffic flow models with anisotropic kernel}}, {ESAIM: Mathematical Modelling and Numerical Analysis}, 52 (2018), pp.~163--180.

\bibitem{GoatinPuppo2024}
{\sc I.~Ciaramaglia, P.~Goatin, and G.~Puppo}, {\em Non-local traffic flow models with time delay: well-posedness and numerical approximation}, 2024.

\bibitem{ColomboGaravelloMercier2012}
{\sc R.~M. Colombo, M.~Garavello, and M.~L\'{e}cureux-Mercier}, {\em A class of nonlocal models for pedestrian traffic}, Mathematical Models and Methods in Applied Sciences, 22 (2012), p.~1150023.

\bibitem{colombo2020local}
{\sc G.~Crippa, E.~Marconi, L.~V. Spinolo, and M.~Colombo}, {\em Local limit of nonlocal traffic models: Convergence results and total variation blow-up}, Annales de l’Institut Henri Poincaré C, Analyse non linéaire, 38 (2021), pp.~1653--1666.

\bibitem{cui2025}
{\sc J.~Cui and R.~Y. He}, {\em Stoch-ident: New method and mathematical analysis for identifying spdes from data}, 2025.

\bibitem{JP_finance}
{\sc F.~Delbaen and H.~Shirakawa}, {\em An interest rate model with upper and lower bounds}, Asia-Pacific Financial Markets, 9 (2002), pp.~191--209.

\bibitem{eberlein}
{\sc A.~Eberle}, {\em Lecture notes on {M}arkov {P}rocesses}, 2015.

\bibitem{Friedrich2021_diss}
{\sc J.~Friedrich}, {\em Traffic flow models with nonlocal velocity}, PhD thesis, University of Mannheim, 11 2021.

\bibitem{Friedrich2018}
{\sc J.~Friedrich, O.~Kolb, and S.~Göttlich}, {\em A {Godunov} type scheme for a class of {LWR} traffic flow models with non-local flux}, Networks and Heterogeneous Media, 13 (2018), pp.~531--547.

\bibitem{garavellohanpiccoli2016book}
{\sc M.~Garavello, K.~Han, and B.~Piccoli}, {\em Models for vehicular traffic on networks}, vol.~9 of AIMS Series on Applied Mathematics, American Institute of Mathematical Sciences (AIMS), Springfield, MO, 2016.

\bibitem{GaravelloPiccoliBook}
{\sc M.~Garavello and B.~Piccoli}, {\em Traffic flow on networks}, vol.~1 of AIMS Series on Applied Mathematics, American Institute of Mathematical Sciences (AIMS), Springfield, MO, 2006.

\bibitem{Garnier2012}
{\sc J.~Garnier, G.~Papanicolaou, and T.-W. Yang}, {\em Anomalous shock displacement probabilities for a perturbed scalar conservation law}, Multiscale Modeling and Simulation, 11 (2013), pp.~1000--1032.

\bibitem{Huang2022}
{\sc K.~Huang and Q.~Du}, {\em Stability of a nonlocal traffic flow model for connected vehicles}, SIAM J. Appl. Math., 82 (2022), pp.~221--243.

\bibitem{Jabari2012}
{\sc S.~E. Jabari and H.~X. Liu}, {\em A stochastic model of traffic flow: Theoretical foundations}, Transportation Research Part B: Methodological, 46 (2012), pp.~156--174.

\bibitem{On_Wright_Fisher_basic_simulation}
{\sc P.~A. Jenkins and D.~Spanò}, {\em Exact simulation of the wright–fisher diffusion}, The Annals of Applied Probability, 27 (2017).

\bibitem{kallenberg2021foundations}
{\sc O.~Kallenberg}, {\em Foundations of Modern Probability}, vol.~99 of Probability Theory and Stochastic Modelling, Springer, Cham, 3~ed., 2021.

\bibitem{keimer2018nonlocal}
{\sc A.~Keimer, L.~Pflug, and M.~Spinola}, {\em Nonlocal scalar conservation laws on bounded domains and applications in traffic flow}, SIAM Journal on Mathematical Analysis, 50 (2018), pp.~6271--6306.

\bibitem{Kruzkov1970}
{\sc S.~N. Kru{\v{z}}kov}, {\em {First} {order} {quasilinear} {equations} {in} {serveral} {independent} {variables}}, Mathematics of the {USSR}-Sbornik, 10 (1970), pp.~217--243.

\bibitem{Levermore_1996}
{\sc C.~Levermore}, {\em Moment closure hierarchies for kinetic theories}, Journal of Statistical Physics, 83 (1996), pp.~1021--1065.

\bibitem{Li2012}
{\sc J.~Li, Q.-Y. Chen, H.~Wang, and D.~Ni}, {\em Analysis of {LWR} model with fundamental diagram subject to uncertainties}, Transportmetrica, 8 (2012), pp.~387--405.

\bibitem{Lions2013}
{\sc P.-L. Lions, B.~Perthame, and P.~E. Souganidis}, {\em Scalar conservation laws with rough (stochastic) fluxes}, Stochastic Partial Differential Equations: Analysis and Computations, 1 (2013), pp.~664--686.

\bibitem{Mishra2012}
{\sc S.~Mishra, N.~H. Risebro, C.~Schwab, and S.~Tokareva}, {\em Numerical solution of scalar conservation laws with random flux functions}, SIAM/ASA Journal on Uncertainty Quantification, 4 (2016), pp.~552--591.

\bibitem{pages}
{\sc G.~Pag{\`e}s}, {\em Numerical Probability}, Universitext, Springer, Cham, 2~ed., 2025.

\bibitem{Risebro2015}
{\sc N.~H. Risebro, C.~Schwab, and F.~Weber}, {\em Multilevel {Monte} {Carlo} front-tracking for random scalar conservation laws}, {BIT} Numerical Mathematics, 56 (2015), pp.~263--292.

\bibitem{Risken1996}
{\sc H.~Risken}, {\em The Fokker--Planck Equation: Methods of Solution and Applications}, vol.~18 of Springer Series in Synergetics, Springer Berlin, Heidelberg, 2~ed., 1996.

\bibitem{Rosini2013}
{\sc M.~D. Rosini}, {\em Macroscopic Models for Vehicular Flows and Crowd Dynamics: Theory and Applications}, Understanding Complex Systems, Springer, Heidelberg, 2013.

\bibitem{Wen2025}
{\sc J.~Wen, J.~Hu, C.~Wu, X.~Xiao, and N.~Lyu}, {\em A novel stochastic second-order macroscopic continuum traffic flow model for traffic instability}, Chaos, Solitons \& Fractals, 190 (2025), p.~115752.

\end{thebibliography}

\begin{appendices}

\section{Proof of Lemma \ref{lem:stab_result}}\label{subsec:proof_lemma_stab}

\begin{proof}
 Since $\gamma_1$ and $\gamma_2$ are arbitrary but fixed, the estimates from \cite{Boehme_2025}, which build on \cite[Theorem 2.4]{Friedrich2018}, apply directly :
    \begin{enumerate}
        \item[(a)] $0 \leq V_{{\gamma_1}}, V_{{\gamma_2}} \leq v^{max} + \tau$, since \newline
         $0 < W_\eta, W_0=1$ and $0 \leq v_{\gamma_i} \leq v^{max}+ \tau$ by construction.
        \item[(2)] $\left\lvert \partial_x V_{{\gamma_1}}(t,x)\right\rvert , \left\lvert \partial_x V_{{\gamma_2}}(t,x)\right\rvert < \infty$, with analogous bounds as in \cite[Theorem A.1.5]{Boehme_2025}.
        \item[(3)] $V_{{\gamma_1}}$ and $V_{{\gamma_2}}$ are Lipschitz continuous with respect to $x$.
    \end{enumerate}

    By construction and considerations (1)-(3), $V_{{\gamma_1}}$ and $V_{{\gamma_2}}$ satisfy the assumptions of Kru\v{z}kov
    \cite{Kruzkov1970}, allowing us 
    to apply the doubling of variables technique. Thus, as in \cite{Friedrich2018} or subsequently \cite[Eq. A.24]{Boehme_2025} we 
    obtain
    \begin{align}\label{eq:2.3.2}
        \qquad {\left\lVert \rho_1(t,\cdot)-\rho_2(t,\cdot) \right\rVert}_{L^1(\mathbb{R})} & \leq  {\left\lVert \rho_1^0-\rho_2^0 \right\rVert}_{L^1(\mathbb{R})} \nonumber
        \\ & \quad + \int_{0}^{T} \int_{\mathbb{R}}  |\partial_x \rho_1(t,x)| \left\lvert V_{\gamma_2}(t,x)-V_{\gamma_1}(t,x)\right\rvert   \,dx  \,dt   \nonumber \\
        &\quad  +\int_{0}^{T} \int_{\mathbb{R}}  |\rho_1(t,x)| \left\lvert \partial_x V_{\gamma_2}(t,x)- \partial_x V_{\gamma_1}(t,x)\right\rvert   \,dx  \,dt ,  
    \end{align}
    where $\partial_x \rho$ has to be understood in the sense of distributions.
     Invoking the Lipschitz-continuity of $v_\gamma$ in both variables, we derive
    \begin{align*}
        \left\lvert v_{\gamma_1}(\rho_1)-v_{\gamma_2}(\rho_2)\right\rvert \leq 
        \lVert v' \rVert_{\infty}  \left\lvert \rho_1 - \rho_2 \right\rvert 
        + \left\lvert \gamma_1 - \gamma_2 \right\rvert, 
    \end{align*}
     such that we can bound the velocity terms in \eqref{eq:2.3.2} by
    \begin{align}\label{eq:LIP}
        |V_{\gamma_2}(t,x)-V_{\gamma_1}(t,x)| &= \Bigl|\int_x^{x+\eta} W_\eta(y-x) \left (v_{\gamma_1}\bigl(\rho_1(t,y),t\bigr) - v_{\gamma_2}\bigl(\rho_2(t,y),t\bigr) \right) \,dy \,\Bigr| \nonumber \\
        &\leq W_\eta(0) \left( \lVert v' \rVert_{\infty}  \|\rho_1(t,\cdot)-\rho_2(t,\cdot)\|_{L^1(\mathbb{R})} +  \left\lvert \gamma_1(t) - \gamma_2(t) \right\rvert \right).
    \end{align}

    Similarly, differentiating $V_{\gamma_2}$ and $V_{\gamma_1}$, we obtain 
    \begin{align}\label{eq:LIP_D}
  &  \nonumber  |\partial_x V_{\gamma_2}(t,x)-\partial_x V_{\gamma_1}(t,x)| \\ & \qquad \leq
     \Big| \int_{x}^{x+\eta} W_\eta'(y-x) \left( v_{\gamma_1}\bigl(\rho_1(t,y),t\bigr) - v_{\gamma_2}\bigl(\rho_2(t,y),t\bigr) \right)  \,dy \nonumber \nonumber \\
            &\qquad\qquad  + W_\eta(\eta) \left( v_{\gamma_1}\bigl(\rho_1(t,x+\eta\bigr),t) - v_{\gamma_2}\bigl(\rho_2(t,x+\eta\bigr),t) \right) \nonumber \\
            &\qquad\qquad - W_\eta(0) \left( v_{\gamma_1}\bigl(\rho_1(t,x),t\bigr) - v_{\gamma_2}\bigl(\rho_2(t,x),t\bigr)\right) \Big|  \nonumber \\
            & \qquad \leq \lVert W_\eta' \rVert_{\infty} \left( \lVert v' \rVert_{\infty}  {\left\lVert \rho_1(t,\cdot)-\rho_2(t,\cdot) \right\rVert}_{L^1(\mathbb{R})} + \left\lvert \gamma_1(t) - \gamma_2(t) \right\rvert \right) \nonumber \\
            &\qquad\qquad + W_\eta(0) \left( \lVert v' \rVert_{\infty}  \left\lvert \rho_1(t,x+\eta)-\rho_2(t,x+\eta) \right\rvert + \left\lvert \gamma_1(t) - \gamma_2(t) \right\rvert \right) \nonumber \\
            &\qquad\qquad + W_\eta(0) \left( \lVert v' \rVert _{\infty} \left\lvert \rho_1(t,x)-\rho_2(t,x) \right\rvert + \left\lvert \gamma_1(t) - \gamma_2(t) \right\rvert \right) \nonumber \\
            & \qquad  = \lVert W_\eta' \rVert_{\infty} \lVert v' \rVert_{\infty}  {\left\lVert \rho_1(t,\cdot)-\rho_2(t,\cdot) \right\rVert}_{L^1(\mathbb{R})} \nonumber \\
            &\qquad\qquad+ W_\eta(0) \lVert v' \rVert_{\infty} \left( 
            \left\lvert \rho_1(t,x+\eta)-\rho_2(t,x+\eta) \right\rvert +
            \left\lvert \rho_1(t,x)-\rho_2(t,x) \right\rvert
            \right) \nonumber \\
            &\qquad\qquad+\left( \lVert W_\eta' \rVert_{\infty} + 2 W_\eta(0)\right) \left\lvert \gamma_1(t) - \gamma_2(t) \right\rvert.
        \end{align}

    Next, we plug our bounds \eqref{eq:LIP} and \eqref{eq:LIP_D} into (\ref{eq:2.3.2}) and obtain (dropping the $\infty$-subscript for notational simplicity) that
\begin{align*}
    &{\left\lVert \rho_1(t,\cdot)-\rho_2(t,\cdot) \right\rVert}_{L^1(\mathbb{R})} \\ 
    &\quad\quad\leq  {\left\lVert \rho_1^0-\rho_2^0 \right\rVert}_{L^1(\mathbb{R})} 
    &&+ W_\eta(0) \lVert v' \rVert \int_{0}^{T}  {\left\lVert \rho_1(t,\cdot)-\rho_2(t,\cdot) \right\rVert}_{L^1(\mathbb{R})} 
    \int_{\mathbb{R}} |\partial_x \rho_1(t,x)|  \,dx  \,dt \\
    &&&+  W_\eta(0) \int_0^T |\gamma_1(t)-\gamma_2(t)| \int_{\mathbb{R}} |\partial_x \rho_1(t,x)|  \,dx \,dt  \\
    &&&+\lVert W_\eta' \rVert \lVert v' \rVert \int_{0}^{T}  {\left\lVert \rho_1(t,\cdot)-\rho_2(t,\cdot) \right\rVert}_{L^1(\mathbb{R})} 
    \int_{\mathbb{R}} |\rho_1(t,x)|  \,dx  \,dt \\
    &&&+ W_\eta(0) \lVert v' \rVert \int_{0}^{T} \int_{\mathbb{R}} \Bigl( |\rho_1-\rho_2|(t,x+\eta)+|\rho_1-\rho_2|(t,x) \Bigr) |\rho_1(t,x)|  \,dx  \,dt\\
    &&&+ \left( \lVert W_\eta' \rVert + 2 W_\eta(0) \right) \int_0^T |\gamma_1(t)-\gamma_2(t)| \int_{\mathbb{R}} |\rho_1(t,x)|  \,dx  \,dt  \\
    &\quad\quad\leq  {\left\lVert \rho_1^0-\rho_2^0 \right\rVert}_{L^1(\mathbb{R})} 
    &&+ W_\eta(0) \lVert v' \rVert  \sup_{t \in [0,T]} \int_{\mathbb{R}} |\partial_x \rho_1(t,x)|  \,dx \int_{0}^{T} {\left\lVert \rho_1(t,\cdot)-\rho_2(t,\cdot) \right\rVert}_{L^1(\mathbb{R})}
      \,dt \\
    &&&+ W_\eta(0)  \sup_{t \in [0,T]} \int_{\mathbb{R}} |\partial_x \rho_1(t,x)|  \,dx \int_0^T |\gamma_1(t)-\gamma_2(t)| \,dt \\
    &&&+\lVert W_\eta' \rVert \lVert v' \rVert \sup_{t \in [0,T]} \int_{\mathbb{R}} |\rho_1(t,x)|  \,dx \int_{0}^{T}  {\left\lVert \rho_1(t,\cdot)-\rho_2(t,\cdot) \right\rVert}_{L^1(\mathbb{R})} 
      \,dt \\
    &&&+ W_\eta(0) \lVert v' \rVert \sup_{t \in [0,T]} \int_{\mathbb{R}} |\rho_1(t,x)|  \,dx \\
    &&&\qquad \cdot    \int_{0}^{T} \int_{\mathbb{R}} \Bigl( |\rho_1-\rho_2|(t,x+\eta)+|\rho_1-\rho_2|(t,x) \Bigr) \,dx  \,dt   \\
    &&&+ \left( \lVert W_\eta' \rVert + 2 W_\eta(0) \right) \sup_{t \in [0,T]} \int_{\mathbb{R}} |\rho_1(t,x)|  \,dx \int_0^T |\gamma_1(t)-\gamma_2(t)|   \,dt. 
\end{align*}

Now since 
\begin{align*}
    \sup_{t \in [0,T]} \int_{\mathbb{R}} |\partial_x \rho_1(t,x)|  \,dx &\leq \sup_{t \in [0,T]} {\left\lVert \rho_1(t,\cdot)\right\rVert}_{\text{BV}(\mathbb{R})} \\
    \sup_{t \in [0,T]} \int_{\mathbb{R}} |\rho_1(t,x)|  \,dx  &= \sup_{t \in [0,T]} {\left\lVert \rho_1(t,\cdot)\right\rVert}_{L^1(\mathbb{R})}, \\
    \int_{0}^{T} \int_{\mathbb{R}} \Bigl( |\rho_1-\rho_2|(t,x+\eta)+|\rho_1-\rho_2|(t,x) \Bigr) \,dx  \,dt &= 2\int_{0}^{T} {\left\lVert \rho_1(t,\cdot)-\rho_2(t,\cdot)\right\rVert}_{L^1(\mathbb{R})} \,dt,
\end{align*}

we obtain the bound
\begin{align*}
        {\left\lVert \rho_1(t,\cdot)-\rho_2(t,\cdot) \right\rVert}_{L^1(\mathbb{R})} 
        \leq  {\left\lVert \rho_1^0-\rho_2^0 \right\rVert}_{L^1(\mathbb{R})} 
        & + K^T_v \int_0^T {\left\lVert \rho_1(t,\cdot)-\rho_2(t,\cdot) \right\rVert}_{L^1(\mathbb{R})} \,dt \\
        & + K^T \int_0^T \left\lvert \gamma_1(t) - \gamma_2(t) \right\rvert \,dt
\end{align*}
with
\begin{align*}
        & K^T_v := \left\lVert v' \right\rVert 
        \left( W_\eta(0) 
            \Biggl( 2  
                \sup_{t \in [0,T]} {\left\lVert \rho_1(t,\cdot)\right\rVert}_{L^1(\mathbb{R})} 
                +\sup_{t \in [0,T]} {\left\lVert \rho_1(t,\cdot)\right\rVert}_{\text{BV}(\mathbb{R})} 
        \right) \\
        & \qquad\qquad\qquad + \lVert W_\eta' \rVert \sup_{t \in [0,T]} {\left\lVert \rho_1(t,\cdot)\right\rVert}_{L^1(\mathbb{R})} \Biggr),\nonumber \\ 
        &K^T :=W_\eta(0) 
      \left(  2 \sup_{t \in [0,T]} {\left\lVert \rho_1(t,\cdot)\right\rVert}_{L^1(\mathbb{R})} +
      \sup_{t \in [0,T]} {\left\lVert \rho_1(t,\cdot)\right\rVert}_{\text{BV}(\mathbb{R})}   
      \right) \\
      &\qquad\qquad\qquad+ \lVert W_\eta' \rVert \sup_{t \in [0,T]} {\left\lVert \rho_1(t,\cdot)\right\rVert}_{L^1(\mathbb{R})}  = \dfrac{K^T_v}{\left\lVert v' \right\rVert } .
\end{align*}
Since all terms term are non-decreasing, we conclude with Gr\"onwalls Lemma that
    \begin{align*}
        {\left\lVert \rho_1(t,\cdot)-\rho_2(t,\cdot) \right\rVert}_{L^1(\mathbb{R})} \leq  \exp(T K^T_v) \left(  {\left\lVert \rho_1^0-\rho_2^0 \right\rVert}_{L^1(\mathbb{R})}+ K^T \int_0^T \left\lvert \gamma_1(t) - \gamma_2(t) \right\rvert \,dt\right).
    \end{align*}
Using Theorem \ref{thm:ex_uniq_sNV_extended} we finally have that
$$ \sup_{t \in [0,T]} {\left\lVert \rho_1(t,\cdot)\right\rVert}_{L^1(\mathbb{R})}=\left\lVert \rho_1^{0}\right\rVert_{L^1(\mathbb{R})}$$
as well as
$$  \sup_{t \in [0,T]} {\left\lVert \rho_1(t,\cdot)\right\rVert}_{\text{BV}(\mathbb{R})}  \leq \exp(TC_1)TV(\rho_0;\mathbb{R}),$$
 which finishes the proof.
\end{proof}  
\end{appendices}

\end{document}